\title{\vspace{-1cm}Global well-posedness of strong solutions of Doi model with large viscous stress}
\date{\today}
\documentclass[12pt]{article}
\usepackage{authblk}
\usepackage{blindtext}
\usepackage{amsfonts}
\usepackage{amsmath}
\usepackage{amsthm}
\usepackage{graphicx}
\usepackage[margin=2cm]{geometry}
\newtheorem{theorem}{Theorem}

\newtheorem{proposition}{Proposition}
\newtheorem{remark}{Remark}

\newtheorem{definition}{Definition}
\newcommand{\norm}[1]{\left\lVert#1\right\rVert}
\author[1]{	Joonhyun La \thanks{joonhyun@math.princeton.edu}}
\affil[1] {Department of mathematics, Princeton University}
\begin{document}
\maketitle
\bibliographystyle{abbrv}
\abstract{We study models of dilute rigid rod-like polymer solutions. We establish the global well-posedness the Doi model for large data, and for arbitrarily large viscous stress parameter. The main ingredient in the proof is the fact that the viscous stress adds dissipation to high derivatives of velocity.}

\section{Introduction}

Study of rod-like polymer suspensions has various applications. In particular liquid crystals are successfully modeled as rigid rod. We are interested in a dilute suspension of rigid rod-like polymers, in dimension 2. In particular, we investigate the Doi model:
\begin{equation}
\begin{gathered}
\partial_t u + u \cdot \nabla_x u = - \nabla_x p + \Delta_x u + \nabla_x \cdot \sigma, \\
\nabla_x \cdot u =0, \\
\partial_t f + u \cdot \nabla_x f = k \Delta_m f + \nu \Delta_x f - \nabla_m \cdot \left ( P_{m^\perp} \left ( (\nabla_x u ) m f \right ) \right ), \\
\sigma = 2 \int_{\mathbb{S}^1} ( m \otimes m  - \frac{1}{2}\mathbb{I}_2 ) f dm + \eta \int_{\mathbb{S}^1}  \left ( (\nabla_x u ) : m \otimes m \right ) m \otimes m f dm, \\
(x,m, t) \in \mathbb{T}^2 \times \mathbb{S}^1 \times (0, T), \\
u(x, 0) = u_0 (x), f(x, m, 0 ) = f_0 (x, m),
\end{gathered} \tag{Doi} \label{Doi}
\end{equation}
where $u$ is the velocity field of the fluid, $p$ is the pressure, $\sigma$ is the added stress field due to the presence of polymer, $f = f(x, t, m)$ is the polymer distribution, and $u_0, f_0$ are initial data.  Also constant parameters $k, \nu >0 $ represents configurational and spatial diffusivity of polymers, respectively, and $\eta > 0$ is a constant parameter representing the concentration of the polymers. We prove global well-posedness of strong solution of (\ref{Doi}). The term
$$ P_{m^\perp} ( g \vec{v} ) = (m^\perp \cdot \vec{v} ) g m^\perp $$
is the projection to the tangent space of $\mathbb{S}^1$ at $m$, and $\nabla_m = \partial_\theta $ in local coordinates. The polymer stress tensor $\sigma$ can be decomposed into two terms: $\sigma = \sigma_E + \sigma_V$, where
\begin{equation}
\sigma_E (f) = 2 \int_{\mathbb{S}^1} (m \otimes m - \frac{1}{2} \mathbb{I}_2 ) f dm, \tag{Elastic} \label{Elastic} 
\end{equation}
and
\begin{equation}
\sigma_V (f) = \eta \int_{\mathbb{S}^1} ((\nabla_x u) : m \otimes m ) m \otimes m f dm. \tag{Viscous} \label{Viscous}
\end{equation}
The presence of viscous stress tensor is the main difficulty for the well-posedness of the Doi model. Viscous stress tensors arise from rigidity constraint of the polymer( \cite{Doi1986}), and mathematically $\sigma_V (f)$ is not elliptic in $u$, which makes the momentum equation of (\ref{Doi}) non-parabolic for large $\eta$. This difficulty can be clearly illustrated in the approximate Doi model:
\begin{equation}
\begin{gathered}
\partial_t u + u \cdot \nabla_x u = - \nabla_x p + \Delta_x u + \nabla \cdot \sigma, \\
\nabla_x \cdot u = 0, \\
\sigma = \eta (\nabla_x u : A ) A, \\
\partial_t A + u \cdot \nabla_x A  = (\nabla_x u ) A + A (\nabla_x u) ^T - 2 (\nabla_x u : A) A - 2k (2A - \mathbb{I}_2 ) + \nu \Delta_x A, \\
u(x, 0) = u_0 (x), A(x, 0) = A_0 (x), \\
(x, t) \in \mathbb{T}^2 \times (0, T).
\end{gathered} \tag{DA} \label{DA}
\end{equation}
The model (\ref{DA}) is an approximate closure of Doi model (\ref{Doi}) obtained by letting $A = \int_{S^1} n \otimes n f dn$ and adopting the decoupling approximation $\sigma \simeq \eta (\nabla_x u : A) A$ and ignoring elastic stress part. We establish the energy estimate:
\begin{equation}
\frac{1}{2} \frac{d}{dt} \norm{u}_{L^2} ^2 + \norm{\nabla_x u}_{L^2} ^2 + \eta \int | ( \nabla_ x u) : A |^2 dx = 0,
\end{equation}
and we see that in fact viscous stress is another dissipative structure for $u$. Based on this remarkable property, which holds in (\ref{Doi}) also, Lions and Masmoudi proved global existence of weak solution of (\ref{Doi}) in \cite{MR2340887}. However, when we apply the vorticity estimate, at first point we can only obtain
\begin{equation}
\frac{d}{dt} \norm{\omega}_{L^2} ^2 + \norm{\nabla_x \omega}_{L^2} ^2 \le \norm{\nabla_x \cdot \sigma}_{L^2} ^2 \le \eta ^2 \left (\norm{\nabla \nabla u }_{L^2} ^2 + (\mathrm{Error}) \right )
\end{equation}
which makes the right hand side for the second inequality intractable if $\eta > \eta_c$ for some threshold $\eta_c$. 
\newline \newline
Recently, in \cite{musacchio2018enhancement}, the authors numerically discovered that, when $\eta$ exceeds some threshold $\eta_c$, the flow governed by (\ref{DA}) becomes chaotic. It was hence unclear this phenomenon supports the claim that the systems (\ref{Doi}) and (\ref{DA}) lack structure to control higher regularity of $u$. However, in this work we find that actually the viscous stress tensor adds dissipation for higher derivatives of $u$ also, modulo derivatives in polymer variables ( (\ref{C-Doi2D}), (\ref{C2-Doi2D})). This observation is crucial in proving global well-posedness of (\ref{Doi}) and (\ref{DA}) in diffusive systems $\nu > 0$.

\paragraph{Notion of the solution.}
For the notion of solution, we follow the argument in \cite{jl1}. By focusing on the evolution of macroscopic variables (trigonometric moments in this case), we can set up well-posedness of strong solutions for large class of initial data. In particular, higher regularity of Fokker-Planck equation is not necessary, and weak solution for Fokker-Planck equation is sufficient. On the other hand, since the effect of polymer to the flow are characterized by stresses, which are moments in (\ref{Doi}), requiring spatial regularity for appropriate moments is necessary. In this regard, we introduce a terminology: for any $n \in \mathbb{Z}_{>0}$, we let
\begin{equation}
M_n (x, t) := \left ( M_n ^I (x, t) \right )_{I : |I| = n } := \left ( \int_{\mathbb{S}^1} m^I f (x, t, m) dm \right )_{I : |I| = n} \tag{Moment} \label{Moment}
\end{equation}
be the vector of all moments of $f$ of order $n$. Also, we define the weak solution as following (\cite{MR3443169}):
\begin{definition}
Given a divergence-free vector field $v \in L^\infty (0, T; W^{2,2} ) \cap L^2 (0, T; W^{3,2} )$, $\mu$ is a weak solution to the Cauchy problem
$$ \partial_t \mu + v \cdot \nabla_x \mu = k \Delta_m \mu + \nu \Delta_x \mu - \nabla_m \cdot (P_{m^\perp}  ( ( \nabla_x v) m \mu ) ) , \mu (t=0) = \nu $$
if for almost every $t \in (0, T)$,
\begin{equation}
\begin{gathered}
\int_{\mathbb{T}^2 \times \mathbb{S}^1 } \phi_x (x) \phi_m (m) d\mu (x, t; dm) dx - \int_{\mathbb{T}^2 \times \mathbb{S}^1 } \phi_x (x) \phi_m (m) d \nu (x; dm) dx \\
= \lim_{\tau \rightarrow 0} \int_{\tau} ^t \int_{\mathbb{T}^2 \times \mathbb{S}^1 } \left [ v \cdot \nabla_x \phi_x \phi_m + k \Delta_m \phi_m \phi_x + \nu \Delta_x \phi_x \phi_m + \phi_x \nabla_m \phi \cdot P_{m^\perp} ( (\nabla_x v ) m ) \right ] \mu(x; s; dm) dxds
\end{gathered} \tag{Cauchy-FP} \label{Cauchy-FP}
\end{equation}
for every $\phi_x \in C^\infty (\mathbb{T}^2 )$, $\phi_m \in C^\infty (\mathbb{S}^1)$. 
\end{definition}
Our main result is the following:
\begin{theorem} \label{Main}
Suppose that $u_0 \in \mathbb{P}W^{2,2} (\mathbb{T}^2)$, $f_0 \ge 0 \in L^1(\mathbb{T}^2 \times \mathbb{S}^1 )$ with $\sigma_E(f_0 ) \in W^{1,2} (\mathbb{T}^2 ), \int_{\mathbb{T}^2} \int_{\mathbb{S}^1} (f_0 \log f_0 -f_0 +1 ) dm dx < \infty$, $M_0 \in L^\infty (\mathbb{T}^2 )$, $M_4 ( f_0) \in W^{2,2} (\mathbb{T}^2)$, and $M_6 (f_0) \in W^{1,2} (\mathbb{T}^2 ) $. Then there is a unique solution $(u, f)$ to (\ref{Doi}), where $u \in L^\infty (0, T; W^{2,2}) \cap L^2 (0, T; W^{3,2})$ is the strong solution of the evolution equation of $u$ for (\ref{Doi}), $\sigma_E (f), M_6 (t) \in L^\infty (0, T; W^{1,2} ) \cap L^2 (0, T; W^{2,2})$, and $M_4(f) \in L^\infty (0, T; W^{2,2} ) \cap L^2 (0, T: W^{3,2} ) $. Also $f$ is given by a density $f(x, t, m)$, and $f$ is a weak solution to the Cauchy problem of the Fokker-Planck equation of (\ref{Doi}). Furthermore, the estimates (\ref{FE}), (\ref{M0B}), (\ref{MnB1}) for $n=4$, (\ref{ESB}), (\ref{V-Doi}), (\ref{M4B2}), and (\ref{V2-Doi}) hold. In addition, $f(t) \in W^{1,1} (\mathbb{T}^2 \times \mathbb{S}^1)$ holds.
\end{theorem}
\paragraph{Previous works}
The system (\ref{Doi}) is an example of more general Fokker-Planck-Navier-Stokes systems. When viscous stress is ignored, global well-posedness is established by various authors including Constantin, Fefferman, Kevrekidis, Masmoudi, Seregin, Titi, Vukadinovic, and Zarnescu (\cite{MR2188682}, \cite{MR2391531}, \cite{MR2276466}, \cite{MR2073948}, \cite{MR2367203}, \cite{MR2600741}, \cite{MR2667634}, \cite{MR2164412}) even in non-diffusive regime. Also Otto and Tzavaras proved global existence of weak solution in 3D Doi model, without viscous stress, coupled with Stokes flow in \cite{MR2365451}. For Doi models (\ref{Doi}), Lions and Masmoudi proved global existence of weak solution in \cite{MR2340887} with important observation of dissipative nature of viscous stress. Also Zhang and Zhang proved local and small data global well-posedness for (\ref{Doi}) models for small $\eta$ in \cite{MR2452886}. Compressible Doi model is discussed by Bae and Trivisa in \cite{MR2974165}, \cite{MR3061144}, and \cite{MR3524331}. The relationship between rigid rod-like polymer suspension models and Ericksen-Leslie model for nematic liquid crystal has been investigated in \cite{MR3366748}. For more general introduction for complex fluids, there are excellent references including \cite{masmoudi2018equations}, \cite{MR2922368}, and \cite{MR3050292}.

\section{Global well-posedness of the strong solution of (\ref{DA})} \label{globalwellsmall}

In this section, we prove the following theorem.
\begin{theorem}
Given $(u_0, A_0 ) \in \mathbb{P}W^{2,2} (\mathbb{T}^2 ) \times W^{2,2} (\mathbb{T}^2 )$, where $A_0$ is a $2 \times 2$ positive definite matrix valued function with $\mathrm{Tr} \, A_0 \equiv 1$, then for any $T>0$ there is a unique strong solution $(u, A) \in \left (L^\infty (0, T; \mathbb{P}W^{2,2} (\mathbb{T}^2) ) \cap L^2 (0, T; \mathbb{P} W^{3,2} (\mathbb{T}^2 ) ) \right ) \times \left (L^\infty (0, T; W^{2,2} (\mathbb{T}^2) ) \cap L^2 (0, T;  W^{3,2} (\mathbb{T}^2 ) ) \right )$ satisfying (\ref{DA}) and $\mathrm{Tr} \, A \equiv 1$ and $A$ remains positive definite. Furthermore, the solution satisfies the estimates (\ref{Energy}), (\ref{A1}), (\ref{A2}), (\ref{Vorticity}), (\ref{A3}), and (\ref{GV}).
\end{theorem}

\subsection{A priori estimates}

First we have the propagation of positive-definiteness and Trace 1 for $A$.
\begin{proposition} \label{PD}
Suppose that $ (\nabla_x  u): A \in L^1 (0, T; L^\infty )$ and $\mathrm{Tr} \, A (0) \equiv 1$ and $A(0)$ is positive definite. Then $A(t)$ remains positive definite with $\mathrm{Tr} \,A \equiv 1$.
\end{proposition}
\begin{proof}
We need to check $\mathrm{det} \, A > 0$ and $\mathrm{Tr} \, A \equiv 1$. For $\mathrm{Tr} \, A \equiv 1 $, we take the trace of the third equation of (\ref{DA}) to get
\begin{equation}
(\partial_t + u \cdot \nabla_x  ) \mathrm{Tr} \, A = 2 \left ( (\nabla_x u : A ) + 2k \right ) (1 - \mathrm{Tr} \, A ) + \nu \Delta_x \mathrm{Tr} \, A 
\end{equation}
and by the maximum principle we are done. For $\mathrm{det} \, A >0$, we have
\begin{equation}
\begin{gathered}
(\partial_t + u \cdot \nabla_x ) \mathrm{det} \, A \\
= -4 ( ( (\nabla_x u ): A) + 2k ) \mathrm{det} \, A + 2k \mathrm{Tr} \, A + \nu \Delta_x \mathrm{det} \, A - 2 \nu \nabla_x A_{11} \cdot \nabla_x A_{22} + 2 \nu | \nabla_x A_{12} |^2 \\
= -4 ( (( \nabla_x u ):A ) + 2k ) \mathrm{det} \, A + \nu \Delta_x \mathrm{det} \, A + (2k + 2 \nu |\nabla_x A_{11} |^2 + 2 \nu |\nabla_x A_{12} | ^2 )
\end{gathered} \tag{Det} \label{Det}
\end{equation}
where we used $\mathrm{Tr} \, A \equiv 1$. Then by the maximum principle we are done again.
\end{proof}
We investigate a priori estimates. First, usual energy estimates give us
\begin{equation}
\frac{1}{2} \frac{d}{dt} \norm{u}_{L^2} ^2 + \norm{\nabla_x u}_{L^2} ^2 = -  \int \nabla_x u : \sigma = - \eta \int | (\nabla_x u) : A |^2 dx, 
\end{equation}
that is,
\begin{equation}
\frac{1}{2} \norm{u}_{L^\infty (0, T; L^2 ) } ^2 + \norm{\nabla_x u}_{L^2 (0,T; L^2 ) } ^2 + \eta \norm{(\nabla_x u):A }_{L^2 (0, T; L^2 ) } ^2 \le \frac{1}{2} \norm{u_0 }_{L^2} ^2. \tag{Energy}\label{Energy}
\end{equation}
For $\norm{A}_{L^\infty (0, T; L^1 \cap L^\infty ) }$, we know that from $\mathrm{Tr }\, A \equiv 1$ and $A$ is positive definite, $\norm{A(t) }_{L^\infty } \le 1 $ for all $t$. Also $0 < \mathrm{det} \, A(x, t)  \le \frac{1}{4}$ is obtained. On the other hand, using (\ref{Det}), we can obtain an estimate for $\norm{\nabla_x A}_{L^2 (0, T; L^2) }$. Integrating (\ref{Det}) with respect to $x$, it can be written as
\begin{equation}
\frac{d}{dt} \int \mathrm{det} \, A dx + 4 \int ((\nabla_x u) : A ) dx + \int 8k \mathrm{det} \, A dx = \nu \norm{\nabla_x A}_{L^2} ^2 + 2k | \mathbb{T}^2 |.
\end{equation}
However, using that $\mathrm{det} \, A \le \frac{1}{4}$, that $|A|^2 = \sum_{ij} A_{ij} ^2 = 1 - 2 \mathrm{det} \, A$ by $\mathrm{Tr} \, A \equiv 1$, and Cauchy-Schwarz inequality we have
\begin{equation}
\frac{1}{2} \frac{d}{dt} \norm{A}_{L^2} ^2 + \nu \norm{\nabla_x A (t) }_{L^2} ^2 \le 4 \norm{( \nabla_x u) : A  (t) }_{L^2} .
\end{equation}
Integrating over time, we have
\begin{equation}
\frac{1}{2} \norm{A}_{L^\infty (0, T; L^2 ) } ^2 + \nu \norm{\nabla_x A}_{L^2 (0, T; L^2 ) } ^2 \le \frac{1}{2} \norm{A (0) }_{L^2} ^2 + C \sqrt{T} \min\left ( \frac{1}{\sqrt{\eta}}, 1 \right ). \tag{A1} \label{A1}
\end{equation}
Also, by multiplying $-\Delta_x A$ to the fourth equation of (\ref{DA}) and integrating we obtain
\begin{equation}
\frac{1}{2} \frac{d}{dt} \norm{\nabla_x A}_{L^2} ^2 + 4k \norm{\nabla_x A}_{L^2} ^2+ \nu \norm{\Delta_x A }_{L^2} ^2 \le \norm{\Delta_x A}_{L^2} \left ( \norm{u}_{L^4} \norm{\nabla_x A}_{L^4} + 4 \norm{\nabla_x u}_{L^2} \right ),
\end{equation}
and by Ladyzhenskaya's inequality
$$ \norm{u}_{L^4} ^2 \le C \norm{u}_{L^2} \norm{\nabla_x u}_{L^2}$$
applied to $\norm{u}_{L^4}$ and $\norm{\nabla_x A}_{L^4}$ and Young's inequality we have
\begin{equation}
\frac{d}{dt} \norm{\nabla_x A}_{L^2} ^2 + 8k \norm{\nabla_x A}_{L^2} ^2 + \nu \norm{\Delta_x A}_{L^2} ^2 \le \frac{C}{\nu^3} \norm{u}_{L^2} ^2 \norm{\nabla_x u}_{L^2} ^2 \norm{\nabla_x A}_{L^2} ^2 +\frac{C}{\nu} \norm{\nabla_x u}_{L^2} ^2,
\end{equation}
so we have, by Gr\"{o}nwall,
\begin{equation}
\begin{gathered}
\norm{\nabla_x A}_{L^\infty (0, T; L^2 )} ^2 + 8k \norm{\nabla_x A}_{L^2 (0, T; L^2 ) } ^2 + \nu \norm{\Delta_x A }_{L^2 (0, T: L^2 ) } ^2 \\
 \le \exp \left ( \frac{C}{\nu^3} \norm{u}_{L^\infty (0, T; L^2 ) } ^2 \norm{\nabla_x u}_{L^2(0, T; L^2 ) } ^2 \right ) \left ( \norm{\nabla_x A(0) }_{L^2} ^2 + \frac{C}{\nu} \norm{\nabla_x u }_{L^2 (0, T: L^2 ) } ^2 \right ) \le C_1,
\end{gathered} \tag{A2} \label{A2}
\end{equation}
where $C_1$ depends only on the norm of initial data. Then we take the curl $(-\partial_2, \partial_1 ) \cdot$ to the velocity equation of (\ref{DA}), multiply $\omega = \partial_1 u_2 - \partial_2 u_1 $, and integrate to obtain
\begin{equation}
\frac{1}{2} \frac{d}{dt} \norm{\omega}_{L^2} ^2 + \norm{\nabla_x \omega} _{L^2} ^2 = \int \omega \nabla_x ^{\perp} \cdot (\nabla_x \cdot \sigma ) dx. \tag{V} \label{V}
\end{equation}
\paragraph{Controlling $\int \omega \nabla_x ^{\perp} \cdot (\nabla_x \cdot \sigma ) dx$.}
Conventional estimate for the term $\int \omega \nabla_x ^{\perp} \cdot (\nabla_x \cdot \sigma ) dx$ is to use Cauchy-Schwarz inequality, which makes the term $\norm{\nabla_x \omega}_{L^2} ^2$ uncontrollable, however a closer look at the term allows us a better estimate. Note that
\begin{equation}
\begin{gathered}
\int \omega \nabla_x ^{\perp} \cdot (\nabla_x \cdot \sigma ) dx = \int \omega \left  ( (\partial_{1} ^2 - \partial_{2} ^2 ) \sigma_{12} + \partial_1 \partial_2 (\sigma_{22} - \sigma_{11} )  \right ) dx \\
= \int (\partial_1 ^2 - \partial_2 ^2 ) \omega \sigma_{12} + \partial_1 \partial_2 \omega (\sigma_{22} - \sigma_{11} ) dx \\
= \eta \int \left ( (\partial_1 ^2 - \partial_2 ^2 ) \omega A_{12} + \partial_1 \partial_2 \omega (A_{22} -A_{11} ) \right ) (\nabla_x u ) : A dx.
\end{gathered}
\end{equation}
Also, note that 
$$(\nabla_x u ) : A = \partial_1 u_1 A_{11} + \partial_1 u_2 A_{12} + \partial_2 u_1 A_{12} + \partial_2 u_2 A_{22}  $$ 
and we introduce the stream function $\psi$, that is, $u = \nabla^{\perp} \psi = (-\partial_2 \psi, \partial_1 \psi).$  Then we have 
\begin{equation*}
\omega = \Delta_x \psi, \,\, - \partial_1 u_1 = \partial_2 u_2 = \partial_1 \partial_2 \psi, \,\,  \partial_1 u_2 = \partial_1 ^2 \psi, \,\, \partial_2 u_1 = - \partial_2 ^2 \psi.
\end{equation*}
Therefore, we have
\begin{equation}
\begin{gathered}
 \int \omega \nabla_x ^{\perp} \cdot (\nabla_x \cdot \sigma ) dx  = \eta \int \left ( (\partial_1 ^2 - \partial_2 ^2 ) \omega A_{12} + \partial_1 \partial_2 \omega (A_{22} -A_{11} ) \right ) (\nabla_x u ) : A dx  \\
=  \eta \int  \left ( \Delta_x (\partial_1 ^2 -\partial_2 ^2 ) \psi A_{12} + \Delta_x \partial_1 \partial_2 \psi (A_{22} - A_{11} ) \right ) \left ( \partial_1 \partial_2 \psi (A_{22} - A_{11} ) + (\partial_1 ^2  - \partial_2 ^2 ) \psi A_{12} \right ) dx \\
= - \eta \int | A_{12} \nabla_x (\partial_1 ^2 - \partial_2 ^2 ) \psi + (A_{22} - A_{11} ) \nabla_x (\partial_1 \partial_2 ) \psi | ^2 dx + I
\end{gathered} \tag{Cancellation} \label{Cancellation}
\end{equation}
where 
\begin{equation}
\begin{gathered}
|I| \le  C \eta \int | \nabla_x (\Delta_x) \psi | |A| |\nabla_x A | dx \le C \eta \norm{\nabla_x \omega}_{L^2} \norm{\nabla_x A}_{L^2}  \le \frac{1}{2} \norm{\nabla_x \omega}_{L^2} ^2 + C \eta^2 \norm{\nabla_x A}_{L^2} ^2
\end{gathered} 
\end{equation}
Applying this to (\ref{V}), we obtain
\begin{equation}
\frac{d}{dt} \norm{\omega}_{L^2} ^2 + \norm{\nabla_x \omega}_{L^2} ^2 + 2 \eta \norm{ \left ( \nabla (\partial_k u ) : A \right )_{k} }_{L^2} ^2 \le C \eta^2 \norm{\nabla_x A}_{L^2} ^2 ,
\end{equation}
and by Gr\"{o}nwall we obtain
\begin{equation}
\begin{gathered}
\norm{\omega}_{L^\infty (0, T; L^2) }^2 + \norm{\nabla_x \omega}_{L^2 (0, T; L^2) } ^2 + 2 \eta \norm{ \left ( \nabla (\partial_k u ) : A \right )_{k} }_{L^2(0, T; L^2 ) } ^2 \\
 \le \norm{\omega (0) }_{L^2} ^2 + C \eta^2 \norm{\nabla_x A}_{L^2 (0, T; L^2 ) } ^2  = C_2,
\end{gathered} \tag{Vorticity} \label{Vorticity}
\end{equation}
where again $C_2$ depends only on the initial data. Then we multiply $(\Delta)^2 A$ to the fourth equation of (\ref{DA}) and integrate to obtain
\begin{equation}
\begin{gathered}
\frac{d}{dt} \norm{\Delta_x A}_{L^2} ^2 + 4k \norm{\Delta_x A}_{L^2} ^2 + \nu \norm{\nabla_x \Delta_x A}_{L^2} ^2 \\
= \int (\Delta_x ^2 A) \left ( - u \cdot \nabla_x A + (\nabla_x u ) A + A (\nabla_x u ) ^T - 2 (\nabla_x u : A ) A \right ) dx.
\end{gathered}
\end{equation}
The first term in the left-hand side is controlled by
\begin{equation*}
\begin{gathered}
\norm{\nabla_x \Delta_x A }_{L^2} \left ( \norm{\nabla_x u}_{L^4} \norm {\nabla_x A}_{L^4} + \norm{u}_{L^4} \norm{\Delta_x A}_{L^4} \right ) \\
 \le \frac{\nu}{4} \norm{\nabla_x \Delta_x A }_{L^2} ^2 +  \norm{\nabla_x A}_{L^2} ^2 \norm{\Delta_x A}_{L^2} ^2 + \frac{C}{\nu^2 } \norm{\nabla_x u}_{L^2} ^2 \norm{\nabla_x \omega}_{L^2} ^2 + \frac{C}{\nu^3} \norm{u}_{L^2} ^2 \norm{\nabla_x u }_{L^2} ^2 \norm{\Delta_x A}_{L^2} ^2.
\end{gathered}
\end{equation*}
The second and the third term is controlled by
\begin{equation*}
\begin{gathered}
\norm{\nabla_x \Delta_x A}_{L^2} \left ( \norm{\nabla_x \omega}_{L^2} \norm{A}_{L^\infty } + \norm {\nabla_x u}_{L^4} \norm{\nabla_x A}_{L^4} \right ) \\
\le \frac{\nu}{4} \norm{\nabla_x \Delta_x A}_{L^2} ^2 + \frac{C}{\nu} \norm{\nabla_x \omega}_{L^2} ^2 +\frac{C}{\nu} \left ( \norm{\nabla_x u}_{L^2} ^2 + \norm{\Delta_x u}_{L^2} ^2 \right ) + \frac{C}{k \nu^2 }\norm{\nabla_x A}_{L^2} ^2 + k \norm{\Delta_x A}_{L^2} ^2 
\end{gathered}
\end{equation*}
and the last term is controlled by the same term, by $\norm{A}_{L^\infty } \le 1$. Therefore, we have
\begin{equation}
\norm{\Delta_x A}_{L^\infty (0, T; L^2 ) } ^2 + \nu \norm{\nabla_x \Delta_x A}_{L^2(0, T; L^2) } ^2 \le C \exp (C_1 + C\eta^2 ) (\norm{\Delta_x A(0) }_{L^2} ^2 + C(1+\eta^4) ) = C_3 \tag{A3} \label{A3}
\end{equation}
where $C_3$ depends only on the initial data, using (\ref{A2}) instead of (\ref{A1}) when controlling $\norm{\nabla_x A}_{L^2 (0, T: L^2 ) } ^2 $. FInally, we multiply $-\Delta_x \omega$ to the vorticity equation and integrating to obtain
\begin{equation*}
\frac{d}{dt} \norm{\nabla_x \omega}_{L^2} ^2 + \norm{\Delta_x \omega}_{L^2} ^2 =  \int \Delta_x \omega u \cdot \nabla_x \omega + \int (-\Delta_x \omega) \nabla_x ^{\perp} \cdot (\nabla_x \cdot \sigma ) dx.
\end{equation*}
However, by similar calculation to (\ref{Cancellation}), we have
\begin{equation}
\begin{gathered}
\int (-\Delta_x \omega) \nabla_x ^{\perp} \cdot (\nabla_x \cdot \sigma ) dx = - \eta \int | (\nabla_x \Delta_x u ) : A | ^2 dx + I',
\end{gathered} \tag{Cancellation2} \label{Cancellation2}
\end{equation}
where
\begin{equation}
\begin{gathered}
|I'| \le C \eta \norm{\Delta_x \omega}_{L^2} \left ( \norm{\nabla_x \omega}_{L^4} \norm{\nabla_x A}_{L^4} + \norm{\omega}_{L^4}\norm{\Delta_x A}_{L^4} + \norm{\omega}_{L^2} \norm{\nabla_x A}_{L^\infty}^2  \right ) \\
\le C \eta \norm{\Delta_x \omega}_{L^2} ^{\frac{3}{2} } \norm{\nabla_x A}_{L^2} ^{\frac{1}{2} } \norm{\Delta_x A}_{L^2} ^{\frac{1}{2}} \norm{\nabla_x \omega}_{L^2} ^{\frac{1}{2} } + C \eta \norm{\Delta_x \omega} _{L^2} \norm{\Delta_x A}_{L^2} ^{\frac{1}{2} } \norm{\nabla_x \Delta_x A}_{L^2} ^{\frac{1}{2} } \norm{\omega}_{L^2} ^{\frac{1}{2} } \norm{\nabla_x \omega}_{L^2} ^{\frac{1}{2} } \\
+ C \eta \norm{\Delta_x \omega}_{L^2} \norm{\nabla_x A}_{L^2} \norm{ \nabla_x \Delta_x A}_{L^2} \norm{\omega}_{L^2}   \\
\le \frac{1}{4} \norm{\Delta_x \omega}_{L^2} ^2 + C \eta ^3 \norm{\nabla_x A}_{L^2} ^2 \norm{\Delta_x A}_{L^2} ^2 \norm{\nabla_x \omega}_{L^2} ^2 + C \eta ^2 \left ( \norm{\nabla_x A}_{L^2} ^2 \norm{\nabla_x \omega}_{L^2} ^2 + \norm{\nabla_x \Delta_x A}_{L^2} ^2 \norm{\omega}_{L^2} ^2 \right ) \\
+ C \eta ^2 \norm{\nabla_x A}_{L^2} ^2 \norm{\nabla_x \Delta_x A}_{L^2} ^2 \norm{\omega}_{L^2} ^2,
\end{gathered}
\end{equation}
and
\begin{equation}
\int \Delta_x \omega u \nabla_x \omega dx = - \int u \cdot \nabla_x \left ( |\nabla_x \omega|^2 \right ) dx - \int  ( \nabla_x u \nabla_x \omega) \nabla_x \omega dx,
\end{equation}
with Gagliardo-Nirenberg applied to conclude that
\begin{equation}
\left | \int  ( \nabla_x u \nabla_x \omega) \nabla_x \omega dx \right | \le \norm{\nabla_x u}_{L^3} \norm{\nabla_x \omega}_{L^3} ^2 \le \frac{1}{4} \norm{\Delta_x \omega}_{L^2} ^2 + C \norm{\omega}_{L^2} ^2 \norm{u}_{L^2} \norm{\nabla_x \omega}_{L^2} ^2
\end{equation}
To sum up, we have
\begin{equation}
\begin{gathered}
\frac{d}{dt} \norm{\nabla_x \omega}_{L^2} ^2 + \frac{1}{2} \norm{\Delta_x \omega}_{L^2} ^2 + \eta \norm{(\nabla_x \Delta_x u):A }_{L^2} ^2 \\
\le C ( \norm{u_0}_{L^2} \norm{\omega}_{L^2} ^2 +  \eta^3 C_1 \norm{\Delta_x A}_{L^2} ^2 + \eta^2 \norm{\nabla_x A}_{L^2} ^2 ) \norm{\nabla_x \omega}_{L^2} ^2 + C \eta^2 C_1 C_2 \norm{\nabla_x \Delta_x A}_{L^2} ^2 
\end{gathered}
\end{equation}
and by Gr\"{o}nwall we have
\begin{equation}
\begin{gathered}
\norm{\nabla_x \omega}_{L^\infty (0, T; L^2)} ^2 + \frac{1}{2} \norm{\Delta_x \omega}_{L^2 (0, T; L^2 ) } + \eta \norm{(\nabla_x \Delta_x u ) : A }_{L^2 (0, T; L^2 ) }^2  \\
\le \exp \left ( C_4 (1+ \eta^3 C_3 )  \right ) \left ( \norm{\nabla_x \omega (0) }_{L^2 } ^2 + C_5 (1+ \eta^4 ) \right ) 
\end{gathered} \tag{GV} \label{GV}
\end{equation}
where $C_4, C_5$ depend only on the initial data (and parameters except for $\eta$).
\begin{remark}
Same cancellation argument works for the original Doi model (\ref{Doi}), so we can prove global well-posedness of diffusive Doi model for any $\eta >0$. In the presence of an external forcing (applied to the fluid field), the global well-posedness can still be proved by similar estimates.
\end{remark}

\subsection{Local well-posedness}

In this section we prove the local well-posedness. Before we start, we briefly check the difficulty in the usual contraction mapping scheme. We define the Banach space $B = X \times Y$, where
$$ X = L^\infty (0, T_0 ; \mathbb{P}W^{2,2} (\mathbb{T}^2 ) ) \cap L^2 (0, T_0 ; \mathbb{P}W^{3,2} (\mathbb{T}^2 ) ) $$
and 
$$ Y = L^\infty (0, T_0 ; W^{2,2} (\mathbb{T}^2 ) ) \cap L^2 (0, T_0 ; W^{3,2} (\mathbb{T}^2 ) ). $$
We set up a fixed point equation $U = F(U)$ in $B$ for $U = (u, A)$, where $F(U) = (u^{new}, A^{new})$ given by
\begin{equation}
\begin{gathered}
u^{new} (t) = e^{t\Delta_x} u_0 + Q_1 (u, u) + L (u, A), \\
A^{new} (t) = e^{ (\nu \Delta_x - 4k )t } A_0 + Q_2 (u, A) + \frac{1}{2} \left ( 1 - e^{-4kt} \right ) \mathbb{I}_2 ,
\end{gathered}
\end{equation}
where 
\begin{equation}
Q_1 (u, v) = - \int_0 ^t e^{(t-s) \Delta_x } \mathbb{P} (u(s) \cdot \nabla_x v(s) ) ds,
\end{equation}
\begin{equation}
L (u, A) =  \int_0 ^t  e^{(t-s) \Delta_x } \eta \mathbb{P} (\mathrm{div}_x ( ( \nabla_x u (s) : A (s) ) A (s) ) ) ds,
\end{equation}
and
\begin{equation}
Q_2 (u, A) = \int_0 ^t e^{  (t-s) ( \nu \Delta_x - 4k ) }(-u(s) \cdot \nabla_x A(s) + (\nabla_x u (s) ) A(s) + A(s) (\nabla_x u (s) ) ^T - 2 (\nabla_x u(s) : A(s) ) A(s) ) ds.
\end{equation}
We can easily check that 
\begin{equation}
\begin{gathered}
\norm{Q_1 (u, v)}_X \le C \sqrt{T_0} \norm{u}_X \norm{v}_X, \\
\norm{L(u, A)}_X \le C \norm{u}_X \norm{A}_Y, \\
\norm{Q_2 (u, A)}_Y \le C \sqrt{T_0}  \norm{u}_X \left ( \norm{A}_Y  +  \norm{A}_Y ^2 \right ).
\end{gathered}
\end{equation}
For example, if we let $q = Q_2 (u, A)$, then $q$ is the solution of the equation
\begin{equation}
\partial_t q -  \nu \Delta_x q + 4k q = R, \,\, q(0) = 0,
\end{equation}
where 
$$ Q_2 (u, A) = \int_0 ^t e^{(t-s) (\nu \Delta_x - 4k ) } R(s) ds. $$
Then by the standard estimate we obtain 
$$ \norm{q}_{L^\infty (0, T_0; W^{2,2} ) \cap L^2 (0, T_0; W^{3,2} ) } ^2 \le C \norm{R}_{L^2 (0, T; W^{1,2} ) } ^2, $$
and in this case 
\begin{equation}
\begin{gathered}
C \norm{R (s) }_{W^{1,2} }  \le \norm{u(s)}_{L^\infty} \norm{\nabla_x A(s) }_{L^2} + \norm{\nabla_x u(s) }_{L^4} \norm{\nabla_x A(s)}_{L^4} + \norm{u(s)}_{L^\infty} \norm{\Delta_x A(s) }_{L^2} \\
+ \norm{\nabla_x u(s) }_{L^2} \norm{A(s)}_{L^\infty} + \norm{\Delta u(s) }_{L^2} \norm{A(s) }_{L^\infty} + \norm{\nabla_x u(s) }_{L^2} \norm{A(s)}_{L^\infty } ^2 \\
+ \norm{\Delta_x u(s) }_{L^2} \norm{A(s)}_{L^\infty } ^2 + \norm{\nabla_x u(s) }_{L^4} \norm{\nabla_x A(s) }_{L^4} \norm{A}_{L^\infty }  \\
\le C \norm{u(s)}_{W^{2,2} } \norm{A(s)}_{W^{2,2}} (1 + \norm{A(s) }_{W^{2,2} } ) \le C \norm{u}_X \norm{A}_Y (1 + \norm{A}_Y ).
\end{gathered}
\end{equation}
The problem is in $L(u, A)$: to find a contraction mapping we need to guarantee that $F$ is a mapping from a ball $B(0, R) \subset B$ to itself: however, the bounds for $\norm{u^{new} (t)}_{X}$ that we can obtain from this method is $U_0 + \norm{u}_X (C_1 T_0 \norm{u}_X + C_2 \norm{A}_Y)$, and if $\norm{A}_Y \ge \frac{1}{C_2 }$ then this method fails to bound which holds for both $\norm{u}_X$ and $\norm{u^{new} }_X$.  Therefore, instead of contraction mapping principle, we use an approximation scheme for $u$ equation and go with contraction mapping principle for $A$ equation.
\paragraph{Approximation scheme.} Suppose that $u_n \in X$ is given with $\norm{u_n}_X < \infty $ and $u_n (0) = u_0$. We solve
\begin{equation}
\left \{
\begin{gathered}
\partial_t A_n + u_n \cdot \nabla_x A_n = (\nabla_x u_n ) A_n + A_n (\nabla_x u^T) - 2 (\nabla_x u_n : A_n ) A_n + \nu \Delta_x A_n + 2k (2 A_n - \mathbb{I}_2 ), \\
A_n (0) = A_0.
\end{gathered} \right .
\end{equation}
For this equation contraction mapping works well, and local well-posedness is guaranteed, and proposition \ref{PD}, a priori estimates (\ref{A1}), (\ref{A2}), (\ref{A3}) are satisfied except that all the estimates concerning $u$ are replaced by $u_n$. This means that $A_n$ is guaranteed to exist until the time of existence of $u_n$. Our approximation scheme for $u_{n+1}$ is the following.
\begin{equation}
\left \{
\begin{gathered}
\partial_t u_{n+1} + u_{n+1} \cdot \nabla_x u_{n+1} = - \nabla_x p_{n+1} + \Delta_x u_{n+1} +  \nabla_x \cdot \mathcal{J}_{n+1} \left ( (\eta ( \mathcal{J}_{n+1} \left (\nabla_x u_{n+1} \right ) : A_n ) A_n ) \right ), \\
\nabla_x \cdot u_{n+1 } = 0, \,\, u_{n+1} (0) = u_0,
\end{gathered} \right . \tag{UA} \label{UA}
\end{equation}
where $\mathcal{J}_{n+1} f $ is the orthogonal projection of $f$ into space spanned by eigenvectors corresponding to first $(n+1)$-th eigenvalues. Therefore, $\mathcal{J}_{n+1}$ s are symmetric (in fact self-adjoint) and they commute with differentiation. Then we can prove the local well-posedness of the system (\ref{UA}) via contraction mapping, since for the modified polymer-induced nonlinear structure
\begin{equation*}
L^{n+1} (u_{n+1} , A_n) = \int_0 ^t e^{(t-s) \Delta_x} \eta \mathbb{P} ( \nabla_x \cdot \mathcal{J}_{n+1} (( \mathcal{J}_{n+1} (\nabla_x u_{n+1} ) : A_n ) A_n )) (s) ds
\end{equation*}
has the estimate
\begin{equation*}
\norm{L^{n+1} (u_{n+1}, A_n )}_{X} \le \eta (n+1) C \sqrt{T_0} \norm{u_{n+1} }_X \norm{A_n}_{Y} ^2.
\end{equation*}
We then find an estimate of $\norm{u_{n+1} }_{X}$ independent of $n$, which allow us to guarantee existence of the solution $u_{n+1}$ until the time of existence of $u_{n}$, and also the existence of a weak limit of the sequence $\{u_n\}_n$. This estimate can be obtained in the same manner as (\ref{Energy}), (\ref{Vorticity}), and (\ref{GV}), which is essentially the usual energy method together with the cancellation structures (\ref{Cancellation}), (\ref{Cancellation2}), and those estimates hold with the bound depending only on initial data and $T_0$, independent of $n$. Then we have the uniform bounds
$$ \norm{u_n }_X \le D_1, \norm{A_n} _Y \le D_2,$$
so by compactness we have weak limits $u \in X, A \in Y$, and we can check that for some subsequence of $(u_n, A_n)$, again denoted by $(u_n, A_n)$
\begin{equation}
\begin{gathered}
(\nabla_x u_n) A_n \rightarrow ( \nabla_x u) A \,\,\mathrm{in}\,\, L^2(0,T; L^2) \\
(\nabla_x u_n : A_n ) A_n \rightarrow ( \nabla_x u : A ) A \,\,\mathrm{in}\,\, L^2(0,T; L^2)\\
u_n \cdot \nabla_x u_n \rightarrow u \cdot \nabla_x u \,\,\mathrm{in}\,\, L^2(0,T; L^2 ) \\
\end{gathered}
\end{equation}
and
\begin{equation}
\nabla_x \cdot \mathcal{J}_{n} ( (\mathcal{J}_{n} (\nabla_x u_n ):A_n ) A_n ) \rightarrow \nabla_x \cdot ((\nabla_x u  : A ) A) \,\, \mathrm{in} \,\, L^2(0, T; W^{-1, 2} ).
\end{equation}
Note that $u_n, A_n \in L^\infty (0, T_0 ; W^{2,2} )$ are uniformly bounded and $\partial_t u_n, \partial_t A_n \in L^2 (0, T; W^{1,2} ) $ are also uniformly bounded, by Aubin-Lions there is a subsequence of $A_n$ converging to $A$ and $u_n$ converging to $u$ strongly in $C([0, T]; W^{2-\epsilon, 2})$ for small $\epsilon >0$. 
For the first convergence, note that
\begin{equation}
\begin{gathered}
\norm{(\nabla_x u_n ) A_n - (\nabla_x u) A }_{L^2} \le \norm{\nabla_x u_n}_{L^\infty} \norm{A_n - A}_{L^2} + \norm{\nabla_x  (u_n - u)}_{L^2} \norm{A}_{L^\infty} \\
\le D_1 \norm{A_n - A}_{L^2} + \norm{ u_n - u}_{W^{1,2} } D_2
\end{gathered}
\end{equation}
and by Aubin-Lions we are done. Other two can be shown similarly. The last convergence is also straightforward:
\begin{equation}
\begin{gathered}
\mathcal{J}_n \left (  \left ( \mathcal{J}_n \left ( \nabla_x u_n \right ) : A_n \right ) A_n  \right ) - (\nabla_x u : A) A \\
= I_1 + I_2 + I_3 + I_4 + I_5,
\end{gathered}
\end{equation}
where
\begin{equation}
\begin{gathered}
I_1 =  \mathcal{J}_n \left (  \left ( \mathcal{J}_n \left ( \nabla_x u_n \right ) : A_n \right ) (A_n - A)  \right ),  \,\, I_2 =  \mathcal{J}_n \left (  \left ( \mathcal{J}_n \left ( \nabla_x u_n \right ) : A_n - A \right ) A  \right ), \\
I_3 =  \mathcal{J}_n \left (  \left ( \mathcal{J}_n \left ( \nabla_x u_n - \nabla_x u \right ) : A \right ) A  \right ), \,\, I_4 = \mathcal{J}_n \left (  \left ( ( \mathcal{J}_n \left ( \nabla_x u \right ) - \nabla_x u ) : A \right ) A  \right ), \\
I_5 = \mathcal{J}_n \left (  \left ( \left ( \nabla_x u \right ) : A \right ) A  \right ) - (\nabla_x u : A) A.
\end{gathered}
\end{equation}
We have
\begin{equation}
\norm{I_1}_{L^2} , \norm{I_2}_{L^2} \le D_2 \norm{\nabla_x u_n}_{L^\infty} \norm{(A_n - A)}_{L^2} \le D_2 \norm{(A_n - A)}_{L^\infty (0, T; W^{1,2} ) } \norm{ u_n}_{W^{3,2} }.
\end{equation}
By Aubin-Lions lemma, $\norm{I_1}_{L^2 (0, T_0; L^2  ) } + \norm{I_2 }_{L^2 (0, T_0; L^2 ) } \rightarrow 0$ as $n \rightarrow \infty$. $I_3$ can be similarly treated by Aubin-Lions lemma, and $I_4$, $I_5$ can be treated by the property of $\mathcal{J}_n$.
\paragraph{Uniqueness of the solution.} Suppose that $(u, A) \in B$ and $(v, B) \in B $ are two solutions to the initial value problem (\ref{DA}). Then we have
\begin{equation}
\begin{gathered}
\partial_t (u-v) + u \cdot \nabla_x (u-v) + (u-v) \cdot \nabla_x v = -\nabla_x (p_u - p_v ) + \Delta_x (u-v) \\
+ \eta \nabla_x \cdot ( ((\nabla_x u - \nabla_x v ) :A) A + ( (\nabla_x v ) : (A-B)  ) A + ( (\nabla_x v ) : B ) (A-B) ),\\
\partial_t (A-B) + u \cdot \nabla_x (A-B) + (u-v) \cdot \nabla_x B = (\nabla_x (u-v) ) A + (\nabla_x v)(A-B) \\
+ A (\nabla_x (u-v) )^T + (A-B) (\nabla_x v )^T  - 4k (A - B) + \nu \Delta_x (A-B) \\
- 2 ( ((\nabla_x u - \nabla_x v ) :A) A + ( (\nabla_x v ) : (A-B)  ) A + ( (\nabla_x v ) : B ) (A-B) )
\end{gathered}
\end{equation}
and standard relative energy estimate gives
\begin{equation}
\begin{gathered}
\frac{d}{dt} \norm{u-v}_{L^2} ^2 + \norm{\nabla_x (u-v) }_{L^2} ^2 + 2 \eta \norm{\nabla_x (u-v) : A }_{L^2} ^2 \le C \norm{\nabla_x v}_{L^2}  \norm{A-B}_{L^2} \norm{u-v}_{L^2}, \\
\frac{1}{2} \frac{d}{dt} \norm{A-B}_{L^2} ^2 + \nu \norm{\nabla_x (A-B)}_{L^2} ^2 + 4k \norm{A-B}_{L^2} ^2  \\ 
\le C \norm{A-B}_{L^2} \left (\norm{u-v}_{L^2} \norm{\nabla_x A}_{L^2}  +  \norm{A-B}_{L^2} \norm{\nabla_x v}_{L^2} + \norm{ (u-v) }_{L^2} \norm{A}_{W^{1,2}} +  \norm{A-B}_{L^2} \norm{\nabla_x v}_{L^2} \right ) \\
\end{gathered}
\end{equation}
and by a priori estimates on $u, v, A, B$ we have
\begin{equation}
\frac{d}{dt} \left ( \norm{u-v}_{L^2} ^2 + \norm{A-B}_{L^2} ^2 \right ) + \norm{\nabla_x (u-v)}_{L^2} + \nu \norm{A-B}_{L^2} ^2 \le C \left ( \norm{u-v}_{L^2} ^2 + \norm{A-B}_{L^2} ^2 \right ) 
\end{equation}
and by Gr\"{o}nwall inequality $u=v$, $A=B$.

\section{A priori estimate for (\ref{Doi})}

In this section, we establish a priori estimates for (\ref{Doi}). More precisely, we prove the following theorem:
\begin{theorem}
Let $(u, f)$ be a stroong solution of (\ref{Doi}) on $[0, T]$ with initial data satisfying $M_0 (0) \in L^\infty, \sigma_E (0) \in W^{1,2}, M_4 (0) \in W^{2,2}, M_6 (0) \in W^{1,2}$, $\int_{\mathbb{T}^2} \int_{\mathbb{S}^1} \left ( f \log f - f + 1 \right ) (0) dm dx < \infty$, and $u_0 \in \mathbb{P}W^{2,2} $. Then $(u, f)$ satisfies the bounds (\ref{FE}), (\ref{M0B}), (\ref{MnB1}) for $n=4$, (\ref{ESB}), (\ref{V-Doi}), (\ref{M4B2}), and (\ref{V2-Doi}).
\end{theorem}

\subsection{Free energy estimate}

The first one is the well-known free energy estimate.
\begin{equation}
\begin{gathered}
\frac{d}{dt} \left ( \frac{1}{2} \norm{u}_{L^2} ^2 + \int_{\mathbb{T}^2} \int_{\mathbb{S}^1} f \log f - f + 1 dm dx \right ) + k \int_{\mathbb{T}^2} \int_{\mathbb{S}^1} \frac{|\nabla_m f |^2}{f} dm dx  \\
+ \nu \int_{\mathbb{T}^2} \int_{\mathbb{S}^1} \frac{|\nabla_x f|^2}{f} dm dx + \eta \int ( ( \nabla_x u ) : m \otimes m )^2 f dm dx + \norm{\nabla_x u}_{L^2} ^2 dx = 0.
\end{gathered}
\end{equation}
From this we can obtain the bound
\begin{equation}
\norm{u}_{L^\infty (0, T; L^2 ) } ^2 + \sup_{t \in [0, T ] } \int_{\mathbb{T}^2} \int_{\mathbb{S}^1} (f \log f - f + 1 ) (t) dm dx + \norm{\nabla_x u}_{L^2 (0, T; L^2 ) } ^2  + \norm{\nabla_{m, x} \sqrt{f} }_{L^2 (0, T; L^2 (\mathbb{T}^2 \times \mathbb{S}^1)) } ^2  \le B_1 \tag{FE} \label{FE}
\end{equation}
where 
$$ B_1 = C \norm{u_0}_{L^2} ^2 + \int_{\mathbb{T}^2} \int_{\mathbb{S}^1} \left ( f \log f - f + 1 \right ) (0) dm dx $$
with $C$ a constant depending only on parameters $k, \nu$ and $\eta$.

\subsection{Estimate on moments}

In this section, we investigate bounds on moments, which are useful in establishing bounds of elastic and viscous stresses.
\paragraph{Local coordinates.} 
To study the evolution of moments and elastic tensors, it is useful to write the Fokker-Planck equation of (\ref{Doi}) in the local expression. The configuration space $\mathbb{S}^1$ can be represented by $ m(\theta ) = (\cos \theta, \sin \theta)$, and the Fokker-Planck equation of (\ref{Doi}) is
\begin{equation}
\partial_t f + u \cdot \nabla_x f = k \partial_\theta ^2 f + \nu \Delta_x f - \partial_\theta \left ( m(\theta) ^\perp \cdot  \left ( (\nabla_x u ) m (\theta ) \right ) f  \right ) \tag{FPth} \label{FPth}
\end{equation}
where 
$$ m(\theta) ^\perp \cdot ((\nabla_x u) m(\theta ) ) = \frac{1}{2} \cos 2\theta (\partial_1 u_2 + \partial_2 u_1 ) + \frac{1}{2} (\partial_1 u_2 - \partial_2 u_1 ) - \frac{1}{2} \sin 2\theta (\partial_1 u_1 - \partial_2 u_2 ). $$
Also, the expression for elastic stress can be rewritten as:
\begin{equation}
\sigma_E = \int_0 ^{2\pi} \frac{1}{2} f \left ( \cos 2 \theta \left ( \begin{matrix} 1 && 0 \\ 0 && -1 \end{matrix} \right ) + \sin 2 \theta \left ( \begin{matrix} 0 && 1 \\ 1 && 0 \end{matrix} \right ) \right )  d\theta \tag{ESth} \label{ESth}
\end{equation}
and
\begin{equation}
\sigma_V =  \int_0 ^{2\pi} \frac{\eta f }{4} \left ( \cos 2 \theta (\partial_1 u_1 - \partial_2 u_2 ) + \sin 2 \theta (\partial_2 u_1 + \partial_1 u_2 ) \right ) \left (  \mathbb{I}_2 + \cos 2 \theta \left (\begin{matrix}  1 && 0 \\ 0 && -1 \end{matrix} \right ) + \sin 2 \theta \left (\begin{matrix} 0 && 1 \\ 1 && 0 \end{matrix} \right ) \right )  d \theta.  \tag{VSth} \label{VSth}
\end{equation}
\paragraph{Evolution of moments.} The evolution equation for $M_n$, $n > 0$ is derived from (\ref{Doi}):
\begin{equation}
\partial_t M_n + u \cdot \nabla_x M_n = T_{1, n} M_n + \nu \Delta_x M_n + T_{2,n} (\nabla_x u, M_{n+2} ) \tag{ME} \label{ME}
\end{equation}
where $T_{1, n}$ is a constant-coefficient (depending on $n$) matrix and $T_{2, n} (A, B)$ is a constant-coefficient (also depending on $n$) bilinear tensor on $A$ and $B$. On the other hand, when $n=0$, the evolution equation for $M_0$ is given by
\begin{equation}
\partial_t M_0 + u \cdot \nabla_x M_0 = \nu \Delta_x M_0 \tag{M0} \label{M0}
\end{equation}
and from this we obtain
\begin{equation}
\begin{gathered}
\frac{d}{dt} \frac{1}{2} \norm{M_0}_{L^2} ^2 + \nu \norm{\nabla_x M_0 }_{L^2} ^2 = 0, \\
\frac{d}{dt} \norm{\nabla_x M_0}_{L^2} ^2 + \nu \norm{\Delta_x M_0 } _{L^2} ^2 \le C \norm{u}_{L^2} ^2 \norm{\nabla_x u}_{L^2} ^2 \norm{\nabla_x M_0}_{L^2} ^2 \le C B_1 \norm{\nabla_x u}_{L^2} ^2 \norm{\nabla_x M_0 }_{L^2} ^2
\end{gathered}
\end{equation}
and so 
\begin{equation}
\begin{gathered}
\norm{M_0}_{L^\infty (0, T; L^2) } ^2 + \norm{\nabla_x M_0}_{L^2(0, T; L^2) } ^2 \le B_2, \\
\norm{\nabla_x M_0}_{L^\infty (0, T; L^2) } ^2 + \norm{\Delta_x M_0}_{L^2( 0, T; L^2 ) } ^2 \le B_3, \\
\end{gathered} 
\end{equation}
and
\begin{equation}
\norm{M_0}_{L^\infty (0, T; L^\infty ) } \le B_4, \tag{M0B} \label{M0B}
\end{equation}
where
\begin{equation*}
\begin{gathered}
B_2 = C \norm{M_0 (0)}_{L^2} ^2, \,\, B_3 = C \exp (C B_1 ^2 ) \norm{\nabla_x M_0 (0) } _{L^2} ^2, \,\, B_4 = \norm{M_0 (0)}_{L^\infty}
\end{gathered}
\end{equation*}
where again $C$ depends only on parameters, and estimate (\ref{M0B}) follows from the maximum principle. One simple but important observation is the following:
\begin{equation}
| M_n ^I (x, t) | \le M_0 (x, t) \tag{Obs} \label{Obs}
\end{equation}
due to positivity of $f$ and compactness of $\mathbb{S}^1$. By (\ref{Obs}), we obtain estimates for $M_n$, $n>0$; from
\begin{equation}
\begin{gathered}
\frac{1}{2} \frac{d}{dt} \norm{M_n}_{L^2} ^2 + \nu \norm{\nabla_x M_n }_{L^2} ^2 \le C_n \left ( \norm{M_n}_{L^2} ^2 + \norm{\nabla_x u}_{L^2} \norm{M_n}_{L^2} B_4 \right ), \\
\frac{1}{2} \frac{d}{dt} \norm{\nabla_x M_n}_{L^2} ^2 + \nu \norm{\Delta_x M_n}_{L^2} ^2 \le \left ( \norm{u}_{L^4} \norm{\nabla_x M_n}_{L^4}  + C_n  \norm{\nabla_x u}_{L^2}   \right ) \norm{\Delta_x M_n}_{L^2} + C_n \norm{\nabla_x M_n }_{L^2} ^2, \\
\end{gathered}
\end{equation}
where $C_n$ depends only on $n$ and parameters ($k$ in these cases), we have
\begin{equation}
\begin{gathered}
\norm{M_n}_{L^\infty (0, T; L^2 ) } ^2 + \norm{\nabla_x M_n}_{L^2 (0, T; L^2) } ^2 \le B_{5,n}, \\
\norm{\nabla_x M_n}_{L^\infty (0, T; L^2) } ^2 + \norm{\Delta_x M_n}_{L^2 (0, T; L^2 )} ^2 \le B_{6,n}
\tag{MnB1} \label{MnB1}
\end{gathered}
\end{equation}
where
\begin{equation*}
\begin{gathered}
B_{5,n} = C \exp (C_n T) \left ( \norm{M_n (0)}_{L^2} ^2 + B_4 ^2 B_1 \right ), \\
B_{6,n} = C \exp ( C_n T + B_1 ^2 )  \left ( \norm{\nabla_x M_n (0)}_{L^2} ^2 + C_n B_1 \right ).
\end{gathered}
\end{equation*}
Also, similar to estimate (\ref{A3}), we have
\begin{equation}
\begin{gathered}
\frac{d}{dt} \norm{\Delta_x M_n }_{L^2} ^2 + \nu \norm{\nabla_x \Delta_x M_n }_{L^2} ^2 \le C_n \norm{\Delta_x M_n}_{L^2} ^2  \\
+ C_n \norm{\nabla_x \Delta_x M_n}_{L^2} \left ( \norm{u}_{L^4} \norm{\Delta_x M_n}_{L^4} + \norm{\nabla_x u}_{L^4} \norm{\nabla_x M_n}_{L^4} + \norm{\Delta_x u}_{L^2} \norm{M_{n+2} }_{L^\infty} + \norm{\nabla_x u}_{L^4} \norm{\nabla_x M_{n+2} }_{L^4} \right ).
\end{gathered}
\end{equation}
After burying $\norm{\nabla_x \Delta_x M_n}_{L^2}$ term using Cauchy-Schwarz inequality, the first term is bounded by
$$ C_n \norm{u}_{L^2 } ^2 \norm{\nabla_x u}_{L^2} ^2 \norm{\Delta_x M_n}_{L^2} ^2 \le C_n B_1 \norm{\nabla_x u}_{L^2} ^2 \norm{\Delta_x M_n}_{L^2} ^2, $$
the second term is bounded by
$$ C_n \left (\norm{\Delta_x u}_{L^2} ^2 \norm{\nabla_x M_n}_{L^2} ^2 + \norm{\nabla_x u}_{L^2} ^2 \norm{\Delta_x M_n}_{L^2} ^2 \right ) \le  C_n \norm{\nabla_x u}_{L^2} ^2 \norm{\Delta_x M_n }_{L^2} ^2 + C_n B_{6,n} \norm{\Delta_x u}_{L^2 } ^2, $$
the third term is bounded by
$$ C_n B_4 ^2 \norm{\Delta_x u}_{L^2} ^2,$$
and the fourth term is bounded by
$$ C_n \left ( \norm{\nabla_x u}_{L^2} ^2 \norm{\Delta_x M_{n+2} } _{L^2} ^2 + \norm{\nabla_x M_{n+2} }_{L^2} ^2 \norm{\Delta_x u}_{L^2} ^2 \right ) \le C_n B_{6 ,{n+2}} \norm{\Delta_x u}_{L^2} ^2 + C_n \norm{\nabla_x u}_{L^2} ^2 \norm{\Delta_x M_{n+2} }_{L^2} ^2. $$
To sum up, we have
\begin{equation}
\begin{gathered}
\frac{d}{dt} \norm{\Delta_x M_n}_{L^2} ^2 + \nu \norm{\nabla_x \Delta_x M_n}_{L^2} ^2 \le C_n \left ( 1 + (B_1 +1) \norm{\nabla_x u}_{L^2} ^2 \right ) \norm{\Delta_x M_n }_{L^2} ^2 \\
+ C_n \left (  (B_{6,n} + B_4 ^2 + B_{6 ,{n+2 }} ) \norm{\nabla_x u}_{L^2} ^2 + \norm{\nabla_x u}_{L^2} ^2 \norm{\Delta_x M_{n+2} }_{L^2} ^2 \right )
\end{gathered}
\end{equation}
and therefore
\begin{equation}
\norm{\Delta_x M_n}_{L^\infty (0, T; L^2 ) } ^2 + \norm{\nabla_x \Delta_x M_n}_{L^2 (0, T; L^2 ) }  ^2 \le B_{7, n} \left ( B_{8, n} + \norm{\nabla_x u}_{L^\infty (0, T; L^2 ) } ^2 B_{6, n+2 } \right ) \tag{MnB2} \label{MnB2}
\end{equation}
where
\begin{equation*}
\begin{gathered}
B_{7, n} =  C_n \exp \left ( T + (B_1 + 1) B_1 \right ), \,\, B_{8, n} = B_1 (B_{6,n} + B_4 ^2 + B_{6 ,{n+2 }} ) + \norm{\Delta_x M_n (0)}_{L^2} ^2 .
\end{gathered}
\end{equation*}

\subsection{Control of elastic stress}

Elastic stress can be bounded by bounds on $M_2$, since each component of $\sigma_E$ is a component of $M_2$, but we can get better estimates:
\begin{equation}
\partial_t \sigma_E + u \cdot \nabla_x \sigma_E = -4k \sigma_E + \nu \Delta_x \sigma_E + T_{2, 2} ' (\nabla_x u, M_4 )
\end{equation}
where $T_{2,2} ' $ is another constant-coefficient bilinear tensor. Then
\begin{equation}
\begin{gathered}
\norm{\sigma_E}_{L^\infty \cap L^2 (0, T; L^2) } + \norm{\nabla_x \sigma_E}_{L^2 (0, T; L^2)} ^2 \le C (\norm{\sigma_E (0) }_{L^2} ^2 +  B_1 ) = B_9, \\
\norm{\nabla_x \sigma_E}_{L^\infty \cap L^2 (0, T; L^2 ) } + \norm{\Delta_x \sigma_E } _{L^2 (0, T; L^2 ) } ^2 \le C \exp (B_1 ^2 ) \left ( \norm{\nabla_x \sigma_E (0)}_{L^2} ^2 + B_4 ^2 B_1 \right ) = B_{10}. 
\end{gathered} \tag{ESB} \label{ESB}
\end{equation}

\subsection{Control of higher derivatives of $u$ and viscous stress}

We take curl to the Navier-Stokes equation to obtain
\begin{equation}
\partial_t \omega + u \cdot \nabla_x \omega = \Delta_x \omega + \nabla_x ^\perp \cdot \nabla_x \cdot \sigma_E + \eta \nabla_x ^\perp \cdot \nabla_x \cdot \int_{\mathbb{S}^1 } ( (\nabla_x u ) : m \otimes m ) m \otimes m f dm \tag{Vorticity-Doi} \label{Vorticity-Doi}
\end{equation}
and 
\begin{equation}
\begin{gathered}
\frac{1}{2} \frac{d}{dt} \norm{\omega}_{L^2} ^2 + \norm{\nabla_x \omega}_{L^2} ^2 = - \int_{\mathbb{T}^2} \nabla_x ^\perp \omega \cdot (\nabla_x \cdot \sigma_E ) dx + \eta \int_{\mathbb{T}^2} \omega \nabla_x^\perp \cdot \nabla_x  \cdot  \int_{\mathbb{S}^1} ((\nabla_x u ) : m \otimes m ) m \otimes m f dm dx.
\end{gathered}
\end{equation}
We investigate the last term: note that $\omega = \epsilon_{ij} \partial_i u_j$ where $\epsilon_{ij}$ is the Levi-Civita symbol (in this case just $\epsilon_{12} = 1$ and $\epsilon_{21} = -1$) and the last term can be written as:
\begin{equation}
\begin{gathered}
\eta \int_{\mathbb{T}^2} \epsilon_{ij} \partial_i u_j \left ( \epsilon_{k\ell} \partial_k \partial_p \int_{\mathbb{S}^1} ((\nabla_x u ) : m \otimes m ) m_p m_\ell f dm \right ) dx \\
= \eta \int_{\mathbb{T}^2} \int_{\mathbb{S}^1} \left ( \epsilon_{ij} \epsilon_{k \ell} \partial_k \partial_p \partial_i u_j  m_p m_\ell \right ) ((\nabla_x u) : m \otimes m ) f dm dx \\
= \eta \int_{\mathbb{T}^2} \int_{\mathbb{S}^1} \left (  ( \partial_i \partial_p \partial_i u_\ell - \partial_j \partial_k \partial_\ell u_j ) m_p m_\ell \right )((\nabla_x u) : m \otimes m ) f dm dx \\
= \eta \int_{\mathbb{T}^2} \int_{\mathbb{S}^1} \Delta_x  \left ( ( \nabla_x u ) : m \otimes m  \right ) \left ( ( \nabla_x u ) : m \otimes m  \right ) f dm dx
\end{gathered} \tag{C-Doi2D} \label{C-Doi2D}
\end{equation}
and 
\begin{equation}
\begin{gathered}
\eta \int_{\mathbb{T}^2} \int_{\mathbb{S}^1} \Delta_x  \left ( ( \nabla_x u ) : m \otimes m  \right ) \left ( ( \nabla_x u ) : m \otimes m  \right ) f dm dx \\
= - \eta \int_{\mathbb{T}^2} \int_{\mathbb{S}^1} \left | \nabla_x ( (\nabla_x u): m\otimes m ) \right | ^2 f dmdx - \eta \int_{\mathbb{T}^2} T_3 ( \nabla_x \nabla_x u, \nabla_x u, \nabla_x M_4 ) dx
\end{gathered}
\end{equation}
where $T_3$ is a constant-coefficient trilinear form. Therefore,
\begin{equation}
\begin{gathered}
\frac{d}{dt} \norm{\omega}_{L^2} ^2 + \norm{\nabla_x \omega}_{L^2} ^2 + \eta  \int_{\mathbb{T}^2} \int_{\mathbb{S}^1} \left | \nabla_x ( (\nabla_x u): m\otimes m ) \right | ^2 f dmdx \\
\le C \norm{\nabla_x \sigma_E}_{L^2} ^2 + C \eta \norm{\nabla_x \nabla_x u}_{L^2} \norm{\nabla_x u }_{L^4} \norm{\nabla_x M_4} _{L^4} \\
 \le C \norm{\nabla_x \sigma_E}_{L^2} ^2 + \frac{1}{2} \norm{\nabla_x \omega}_{L^2} ^2 + C \norm{\omega}_{L^2} ^2 \norm{\nabla_x M_4}_{L^2} ^2 \norm{\Delta_x M_4}_{L^2} ^2
\end{gathered}
\end{equation}
where $C$ depends only on parameters (the last $C$ is proportional to $\eta^4$) and 
\begin{equation}
\norm{\omega}_{L^\infty (0, T; L^2) } ^2 + \norm{\nabla_x \omega}_{L^2 (0, T; L^2 ) } ^2 \le C \exp \left ( B_{5,4} B_{6,4} \right ) \left ( \norm{\omega(0)}_{L^2} ^2 + B_9 \right ) = B_{11}. \tag{V-Doi} \label{V-Doi}
\end{equation}
Then, by (\ref{V-Doi}) and (\ref{MnB2}) with $n=4$ we have
\begin{equation}
\norm{\Delta_x M_4}_{L^\infty (0, T; L^2 )}^2 + \norm{\nabla_x \Delta_x M_4}_{L^2 (0, T; L^2) } ^2 \le B_{7,4} \left ( B_{8,4} + B_{11} B_{6,6} \right ) = B_{12}. \tag{M4B2} \label{M4B2}
\end{equation}
Finally, by multiplying $- \Delta_x \omega$ to (\ref{Vorticity-Doi}) and integrating, we have
\begin{equation}
\begin{gathered}
\frac{1}{2} \frac{d}{dt} \norm{\nabla_x \omega}_{L^2} ^2 + \norm{\Delta_x \omega}_{L^2} ^2 = \int_{\mathbb{T}^2} u \cdot \nabla_x \omega \Delta_x \omega dx + \int_{\mathbb{T}^2} \Delta_x \omega (\nabla_x ^\perp \cdot \nabla_x \cdot \sigma_E ) dx \\
- \eta \int_{\mathbb{T}^2} \Delta_x \omega \nabla_x^\perp \cdot \nabla_x  \cdot  \int_{\mathbb{S}^1} ((\nabla_x u ) : m \otimes m ) m \otimes m f dm dx.
\end{gathered}
\end{equation}
Again the last term can be rewritten as, by the same calculation to (\ref{C-Doi2D}),
\begin{equation}
\begin{gathered}
\eta \int_{\mathbb{T}^2} \Delta_x \omega \nabla_x^\perp \cdot \nabla_x  \cdot  \int_{\mathbb{S}^1} ((\nabla_x u ) : m \otimes m ) m \otimes m f dm dx \\
= \eta \int_{\mathbb{T}^2} \int_{\mathbb{S}^1} \Delta_x ^2 \left ( ( \nabla_x u ) : m \otimes m  \right ) \left ( ( \nabla_x u ) : m \otimes m  \right ) f dm dx
\end{gathered} \tag{C2-Doi2D} \label{C2-Doi2D}
\end{equation}
and
\begin{equation}
\begin{gathered}
\eta \int_{\mathbb{T}^2} \int_{\mathbb{S}^1} \Delta_x ^2 \left ( ( \nabla_x u ) : m \otimes m  \right ) \left ( ( \nabla_x u ) : m \otimes m  \right ) f dm dx \\
= \eta \int_{\mathbb{T}^2} \int_{\mathbb{S}^1} \left ( \Delta_x \left ( ( \nabla_x u ) : m \otimes m  \right ) \right )^2 f dm dx \\
+ \eta \int_{\mathbb{T}^2} T_4 \left ( \nabla_x \Delta_x u, \nabla_x \nabla_x u, \nabla_x M_4 \right ) dx + \eta \int_{\mathbb{T}^2} T_5 \left ( \nabla_x \Delta_x u, \nabla_x u, \Delta_x M_4 \right ) dx
\end{gathered}
\end{equation}
again $T_4$ and $T_5$ are constant-coefficient trilinear tensors. Thus,
\begin{equation}
\begin{gathered}
\frac{d}{dt} \norm{\nabla_x \omega}_{L^2} ^2 + \norm{\Delta_x \omega}_{L^2} ^2 + \eta \int_{\mathbb{T}^2} \int_{\mathbb{S}^1} \left ( \Delta_x \left ( ( \nabla_x u ) : m \otimes m  \right ) \right )^2 f dm dx \\
\le C ( \norm{u}_{L^2} ^2 \norm{\nabla_x u }_{L^2} ^2   + \norm{\nabla_x M_4}_{L^2} ^2 \norm{\Delta_x M_4}_{L^2} ^2 + \norm{\nabla_x \Delta_x M_4}_{L^2} ^2  )  \norm{\nabla_x \omega}_{L^2} ^2 \\
+ C (  \norm{\Delta_x \sigma_E }_{L^2 } ^ 2 + \norm{\nabla_x u}_{L^2} ^2 \norm{\Delta_x M_4}_{L^2} ^2  ) 
\end{gathered}
\end{equation}
and
\begin{equation}
\begin{gathered}
\norm{\nabla_x \omega}_{L^\infty (0, T; L^2 ) } ^2 + \norm{\Delta_x \omega}_{L^2 (0, T; L^2) } ^2 \le B_{!3}
\end{gathered} \tag{V2-Doi} \label{V2-Doi}
\end{equation}
where
$$B_{13} = C \exp \left ( C ( B_1 ^2 + B_{6,4} ^2 + B_{12} ) \right ) \left ( \norm{\nabla_x \omega (0)}_{L^2} ^2 + B_{10} + B_{11} B_{6,4} \right ). $$
\begin{remark}
The cancellation structures (\ref{C-Doi2D}) and (\ref{C2-Doi2D}) hold for 3D case also. To illustrate, we have
\begin{equation}
\begin{gathered}
\int dmdx \nabla_x \wedge u \cdot \nabla_x \wedge \left ( \nabla_x \cdot \left ( ( (\nabla_x u): m \otimes m ) m \otimes m f \right ) \right ) \\
= \int dmdx \epsilon_{i j' k' } \partial_{j'} u_{k'} \epsilon_{ijk} \partial_j \partial_\ell   ( ( (\nabla_x u ) : m \otimes m ) m_\ell m_k f ) \\
= \int dmdx \epsilon_{ij'k'} \epsilon_{ijk} (\partial_j \partial_j' \partial_\ell u_{k'} ) m_\ell m_k ((\nabla_x u) : m \otimes m ) f  \\
= \int dmdx (\partial_j ^2 \partial_\ell u_k - \partial_j \partial_k \partial_\ell u_j ) m_\ell m_k ((\nabla_x u) : m \otimes m ) f \\
= \int dmdx (\Delta_x ((\nabla_x u):m \otimes m ) ) ((\nabla_x u) : m \otimes m ) f.
\end{gathered}
\end{equation}
\end{remark}

\section{Local well-posedness of (\ref{Doi})}

In this section, we prove local well-posedness of (\ref{Doi}). Once local well-posedness is established, global well-posedness follows from the a priori estimates established in the previous section.

\subsection{Local existence of the solution}
We follow the method presented in Constantin and Seregin: the existence of the system follows from uniform bounds on the approximate system
\begin{equation}
\begin{gathered}
\partial_t u + u \cdot \nabla_x u = -\nabla_x p + \Delta_x u + \nabla_x \cdot \mathcal{J}_\ell ( \sigma_E ) + \nabla_x \cdot \mathcal{J}_\ell \left ( \eta \int_{\mathbb{S}^1} ( \mathcal{J}_\ell (\nabla_x u ) : m \otimes m ) m \otimes m f dm \right ), \\
\nabla_x \cdot u = 0, \\
\partial_t f + \mathcal{J}_\ell (u) \cdot \nabla_x f = k \Delta_m f + \nu \Delta_x f - \nabla_m \cdot ( P_{m^\perp } ( \mathcal{J}_\ell (\nabla_x u ) m f ) )
\end{gathered} \tag{Doi-Approx} \label{Doi-Approx}
\end{equation}
which satisfies the same bounds (\ref{FE}), (\ref{M0B}), (\ref{MnB1}) for $n=4$, (\ref{ESB}), (\ref{V-Doi}), (\ref{M4B2}), and (\ref{V2-Doi}), and solutions of these systems are obtained by an implicit iteration scheme, using linear equations in each step of the approximation:
\begin{equation}
\begin{gathered}
\partial_t u_{n+1} + u_n \cdot \nabla_x u_{n+1} = - \nabla_x p_{n+1} + \Delta_x u_{n+1} \\ + \nabla_x \cdot \mathcal{J}_{\ell} (\sigma_E (f_n ) ) + \nabla_x \cdot \mathcal{J}_\ell \left ( \eta \int_{\mathbb{S}^1} \left ((\mathcal{J}_\ell (\nabla_x u_{n+1} ) ) : m \otimes m \right ) m \otimes m f_n dm \right ), \\
\nabla_x \cdot u_{n+1} = 0, \\
\partial_t f_{n+1} + \mathcal{J}_\ell (u_n ) \cdot \nabla_x f_{n+1} = k \Delta_m f_{n+1} + \nu \Delta_x f_{n+1} - \nabla_m \cdot \left ( P_{m^\perp} ( \mathcal{J}_\ell \nabla_x u_n ) m f_{n+1 } \right ).
\end{gathered} \tag{Doi-Approx2} \label{Doi-Approx2}
\end{equation}
Existence of (\ref{Doi-Approx}) follows from standard arguments in Fokker-Planck equation: first each system in (\ref{Doi-Approx2}) has smooth solution (same regularity as in the a priori estimate, uniform bounds in $n$), and therefore we have weakly convergent subsequence $u_n$ converging to $u$ in $L^\infty (0, T; W^{2,2}) \cap L^2 (0, T; W^{3,2} ) $, and by Aubin-Lions and Rellich-Kondrachov we have $u_n \rightarrow u \in L^2 (0, T; W^{2-\epsilon,2})$, which is a strong convergence. Also we establish similar strong convergence in moments, and we establish convergence of evolution equation of $u_n$ to that of $u$, which proves that the limit $u$ is a weak solution of (\ref{Doi-Approx}), and since $u$ has enough regularity it is a strong solution. We also find the limit $f$ of $f_n$, using the results from the trigonometric moment problem. We see that $f$ is a weak solution of (\ref{Doi-Approx}), and that $f$ is given by the density, and the standard theory gives the free energy estimate (\ref{FE}).
\paragraph{Uniform bounds on solutions of (\ref{Doi-Approx2}).}
Suppose that $\norm{u_q}_{L^\infty (0, T; W^{j,2} ) \cap L^2 (0, T; W^{j+1,2} ) } ^2 \le B_{app} ^j$, and $\norm{M_2^q}_{L^\infty (0, T; W^{j, 2} ) \cap L^2 (0, T; W^{j+1, 2} ) } ^2 + \norm{M_4 ^q}_{L^\infty (0, T; W^{j, 2} ) \cap L^2 (0, T; W^{j+1, 2} ) } ^2 \le F_{app} ^j$ for $j=0,1,2$, and $\norm{M_6 ^q}_{L^\infty (0, T: W^{j,2} ) \cap L^2 (0, T; W^{j+1, 2} ) } ^2 \le F_{app} ^j$ for $j=0, 1$, and for all $q \le n$. We will determined the exact values of $B_{app} ^j$ and $F_{app} ^j$ in the subsequent estimates. Then we have
\begin{equation}
\begin{gathered}
\frac{d}{dt} \norm{u_{n+1} } _{L^2} ^2 +  \norm {\nabla_x u_{n+1} } _{L^2} ^2 + \eta \int \left ( \mathcal{J}_{\ell} (\nabla_x u_{n+1}  ) : m \otimes m \right ) ^2 f_n dm dx \le C \norm{\sigma_E (f_n ) }_{L^2} ^2 \le C B_4 ^2, 
\end{gathered}
\end{equation}
from the energy estimate, and from this we obtain
\begin{equation}
\norm{u_{n+1}}_{L^\infty (0, T; L^2)} ^2 + \norm{\nabla_x u_{n+1} }_{L^2 (0, T; L^2) } ^2 \le  \norm{u(0)}_{L^2} ^2 + C B_4 ^2 T = B_{app} ^0.
\end{equation}
The vorticity equation becomes
\begin{equation}
\begin{gathered}
\partial_t \omega_{n+1} + u_n \cdot \nabla_x \omega_{n+1} = - \epsilon_{ik} \partial_i u_j ^n \partial_j u_k ^{n+1} + \Delta_x \omega_{n+1} + \nabla_x ^{\perp} \cdot \nabla_x \cdot \mathcal{J}_{\ell} \left ( \sigma_E ( f_n ) \right ) \\
 + \nabla_x ^{\perp} \cdot \nabla_x \cdot \mathcal{J}_\ell \left ( \eta \left ( \left ( \mathcal{J}_\ell ( \nabla_x u_{n+1} ) \right ) : m \otimes m \right ) m \otimes m f_{n} dm  \right )
\end{gathered}
\end{equation}
which leads to the estimate
\begin{equation}
\begin{gathered}
\frac{d}{dt} \norm{\omega_{n+1} }_{L^2} ^2 + \norm{\nabla_x \omega_{n+1} }_{L^2} ^2 + \eta \int \left ( \mathcal{J}_\ell (\nabla_x \nabla_x u_{n+1} ) : m \otimes m \right )^2 f_n dm dx \\ 
 \le C \left ( \norm{\nabla_x \sigma_E (f_n) }_{L^2} ^2 + \norm{\omega_{n+1}} _{L^2} ^2 \left ( \norm{\nabla_x M_4 ^n }_{L^2} ^2 \norm{\Delta_x M_4 ^n }_{L^2} ^2 + 1 \right ) + \norm {\omega_n}_{L^2} ^2 \right ), \\
 \frac{d}{dt} \norm{\nabla_x \omega_{n+1} }_{L^2} ^2 + \norm{\Delta_x \omega}_{L^2} ^2 + \eta \left ( \Delta_x ( (\nabla_x u ) : m \otimes m ) ^2 f_n dm dx\right ) \\
 \le C \left ( \norm{u_n}_{L^2} ^2 \norm{\nabla_x u_n}_{L^2} ^2 + \norm{\nabla_x M_4 ^n}_{L^2} ^2 \norm{\Delta_x M_4 ^n }_{L^2} ^2 + \norm{\nabla_x \Delta_x M_4 ^n}_{L^2} ^2 + \norm{\omega_n}_{L^2} ^2 \right ) \norm{\nabla_x \omega_{n+1} }_{L^2} ^2 \\  
 + C \left ( \norm{\Delta_x \sigma_E (f_n ) }_{L^2} ^2 + \norm{\nabla_x u_{n+1} }_{L^2} ^2 \norm{\Delta_x M_4 ^n }_{L^2} ^2 + \norm{\nabla_x \omega_n}_{L^2} ^2 \norm{\omega_{n+1} }_{L^2} ^2 \right )
\end{gathered}
\end{equation}
Also we have
\begin{equation}
\begin{gathered}
\norm{M_4 ^n }_{L^\infty (0, T; L^2) } ^2 + \norm{ \nabla_x M_4 ^n }_{L^2(0, T; L^2) } ^2 \le e^{CT} (\norm{M_4 ^n (0)}_{L^2} ^2 + B_4 B_{app} ^1 ) \le e^{CT} (\norm{M_4 (0)}_{L^2} ^2 + B_4 B_{app} ^0 ) = F_{app} ^0, \\
\norm{\nabla_x M_4 ^n}_{L^\infty (0, T; L^2 )} ^2 + \norm{\Delta_x M_4 ^n}_{L^2 (0, T; L^2 ) } ^2 \le e^{(B_{app} ^0)^2 + CT } \left ( \norm{ \nabla_x M_4 (0)}_{L^2} ^2 + B_{app} ^0 \right ) = F_{app} ^1, \\
\end{gathered}
\end{equation}
and the same bound for $M_2 ^n $ (or $\sigma_E (f_n ) $) and $M_6 ^n$ in place of $M_4 ^n$. From this we conclude that
\begin{equation}
\norm{\omega_{n+1} }_{L^\infty (0, T; L^2) } ^2 + \norm{\nabla_x \omega_{n+1} }_{L^2} ^2 \le e^{CT + C ( F_{app} ^1 )^2 } \left ( \norm{\omega (0)}_{L^2} ^2 + C F_{app} ^0 + B_{app} ^0 \right ) = B_{app} ^1.
\end{equation}
Then 
\begin{equation}
\norm{\Delta_x M_4 ^n}_{L^\infty (0, T; L^2 )} ^2 + \norm{\nabla_x \Delta_x M_4 ^n}_{L^2 (0, T; L^2) } ^2 \le e^{C ( B_{app} ^0 + 1 )B_{app} ^0 } \left ( \norm{\Delta_x M_4 (0)}_{L^2} ^2 + C ( F_{app} ^1 + B_4 ^2 ) B_{app} ^1\right ) = F_{app} ^2,
\end{equation}
and finally
\begin{equation}
\begin{gathered}
\norm{\nabla_x \omega_{n+1} }_{L^\infty (0, T; L^2) } ^2 + \norm{\Delta_x\omega_{n+1} }_{L^2 (0, T; L^2) } ^2 \\ 
\le e^{C ( ( B_{app} ^0 ) ^2 + (F_{app} ^1 ) ^2 + F_{app} ^2 + B_{app} ^0 ) } ( \norm{\nabla_x \omega (0)}_{L^2} ^2 + F_{app} ^1 + F_{app} ^1 B_{app } ^1 + (B_{app} ^1 ) ^2 ) = B_{app } ^2.
\end{gathered}
\end{equation}
This verifies that $\norm{u_n}_{L^\infty (0, T: W^{2,2} ) \cap L^2 (0, T; W^{3,2} ) }$ is uniformly bounded. Furthermore, $\partial_t u_n$ is also uniformly bounded in $L^2 (0, T; L^2 ) $. 
\paragraph{Convergence of $u_n$ to $u$ and existence of solution for $u$ equation of (\ref{Doi-Approx}). }
By Banach-Alaoglu, we have a subsequence of $u_n$, weakly converging to $u$ in $L^\infty (0, T; W^{2,2} ) \cap L^2 (0, T; W^{3,2} ) $, and by Aubin-Lions, in fact
$$ u_n \rightarrow u \in C([0, T]; W^{2-\epsilon,2} ) \,\, \mathrm{strongly} $$
for small enough $\epsilon > 0$, for a further subsequence. We extract further subsequence that $u_n \rightarrow u$, $\nabla_x u_n \rightarrow \nabla_x u$ almost everywhere. Moreover, we can find a further subsequence such that there is $\sigma_E$ and $M_4$ such that
\begin{equation}
\sigma_E (f_n) \rightarrow \sigma_E, M_4 (f_n) = M_4 ^n \rightarrow M_4 \in C ([0, T]; W^{2-\epsilon, 2 } ) \,\, \mathrm{strongly}
\end{equation}
also. To show that $u$ is a solution of (\ref{Doi-Approx}), we first recall that $W^{2-\epsilon, 2} (\mathbb{T}^2 ) $ is a Banach algebra for $\epsilon < 1$, and also a refined version of Agmon inequality:
\begin{equation}
\norm{u}_{L^\infty (\mathbb{T}^2 ) } \le C \norm{u}_{W^{2-\epsilon, 2} (\mathbb{T}^2 ) }  \tag{Agmon} \label{Agmon}
\end{equation}
Now the evolution equation for $u_{n+1}$ of (\ref{Doi-Approx2}) can be rewritten as the following:
\begin{equation}
\partial_t u_{n+1} = \mathbb{P} \left ( - u_n \cdot \nabla_x u_{n+1} + \Delta_x u_{n+1} + \nabla_x \cdot \mathcal{J}_\ell (\sigma_E (f_n ) ) + \nabla_x \cdot \mathcal{J}_\ell T (\mathcal{J}_\ell (\nabla_x u_{n+1 } ) , M_4 (f_n ) ) \right )
\end{equation}
where $\mathbb{P}$ is the orthogonal projection to divergence-free vector field and $T$ is a constant-coefficient bilinear tensor between two arguments. We first control $u_n \cdot \nabla_x u_{n+1}$. Since $u_n \rightarrow u \in C([0, T]; L^\infty )$ by Agmon and $\nabla_x u_{n+1} \rightarrow \nabla_x u \in C([0, T]; L^2)$, $\mathbb{P} (u_n \cdot \nabla_x u_{n+1} ) \rightarrow \mathbb{P} (u \cdot \nabla_x u ) \in C([0, T]; L^2)$. Also, $\mathbb{P} \Delta_x u_{n+1} \rightarrow \mathbb{P} \Delta_x u \in C([0, T]; W^{-1, 2} ) $, $\mathbb{P} \nabla_x \cdot \mathcal{J}_\ell (\sigma_E (f_n ) ) \rightarrow \mathbb{P} \nabla_x \cdot \mathcal{J}_\ell (\sigma_E) \in C([0, T]; W^{1-\epsilon, 2 } )$ (uniformly in $\ell$), and finally since $M_4 (f_n) \rightarrow M_4 \in C([0, T]; L^\infty )$ by Agmon and $\nabla_x u_{n+1} \rightarrow \nabla_x u \in C([0, T]; L^2 )$, $T ( \mathcal{J}_\ell (\nabla_x u_{n+1}) , M_4 (f_n )) \rightarrow T ( \mathcal{J}_\ell (\nabla_x u) , M_4 ) \in C([0, T]; L^2 ) $ and so  $\nabla_x \cdot \mathcal{J}_\ell T ( \mathcal{J}_\ell (\nabla_x u_{n+1}) , M_4 (f_n )) \rightarrow \nabla_x \cdot \mathcal{J}_\ell T ( \mathcal{J}_\ell (\nabla_x u) , M_4 ) \in C([0, T]; W^{-1, 2} ) $ (uniformly in $\ell$). Finally, since $\partial_t u_n $ is weakly convergent to $\partial_t u$ in $L^2 (0, T; L^2)$, we see that $u$ is a weak solution of $u$-part of (\ref{Doi-Approx}).
\paragraph{Convergence of $f_n$ to $f$.}
To deal with this issue, we recall the result from the trigonometric moment problem:
\begin{theorem}[Carath\'{e}odory-Toeplitz]
For a complex sequence $s = (s_j) \in {\mathbb{N} \cup \{0\} }$ the following are equivalent:
\begin{enumerate}
\item There exists a (nonnegative) radon measure $\mu$ on $\mathbb{T}^1$ such that
$$s_j = \int_{\mathbb{S}^1} e^{-i j \theta} d\mu(\theta) $$
for all $j \in \mathbb{Z}$. Here $s_{-j} := \bar{s_j}$ for $n \ge 1$.
\item $\sum_{j,k = 0} ^\infty s_{j-k} c_k \bar{c_j} \ge 0$ for all finite complex sequences $(c_j)_{j \in  {\mathbb{N} \cup \{0\} }}$.
\end{enumerate}
The measure $\mu$ is uniquely determined by determined by $s_j$.
\end{theorem}
In this regard, we define the trigonometric moments
$$ s_j (f_n) = \int_{\mathbb{S}^1} e^{-ij \theta} f_n (\theta) d \theta. $$
By the very similar argument as before, we obtain
$$ \norm{s_j (f_n)}_{L^\infty (0, T; L^2) } ^2 + \norm{\nabla_x s_j (f_n) }_{L^2} ^2 \le B_4 ^2 (1 + j^2  B_{app} ^0 ), $$
which is uniform in $n$, and $ \norm{\partial_t s_j (f_n)}_{W^{-1, 2} } \le  (\sqrt{B_{app} ^2 } + C \nu ) \norm{\nabla_x s_j (f_n) }_{L^2} + C j B_4 \norm{\nabla_x u}_{L^2} $ so $\partial_t s_j (f_n )$ is uniformly bounded in $L^2 (0, T; W^{-1, 2})$. Therefore, again by Aubin-Lions and diagonalization argument, we can find a further subsequence of $f_n$ such that $s_j (f_n) $ converges to some $s_j$ strongly in $L^2 (0, T; L^2)$, and therefore almost everywhere, for all $j$. Also, we see that
$$ \sum_{j,k=0} ^\infty s_{j-k} (f_n) (x, t) c_k \bar{c_j} \ge 0 $$
for all $x, t$ for all finite complex sequences $(c_j)_{j \in \mathbb{N}\cup \{0 \} }$, and by almost everywhere convergence $ \sum_{j,k=0} ^\infty s_{j-k}  (x, t) c_k \bar{c_j} \ge 0 $ for almost every $(x, t)$. Therefore, we see that there exists a (nonnegative) radon measure $\mu$ such that $s_j = \int_{\mathbb{S}^1} e^{-ij \theta} d\mu (\theta ) $ for all $j \in \mathbb{Z}$. Next we show that for almost every $(x,t)$ $f_n (x, t, m) dm$ converges to $\mu(x, t; dm)$ weakly. The argument is analogous to the method of moment: since $\int_{\mathbb{S}^1} f_n (x, t, m ) dm \le B_4$, $( f_n (x, t) dm)$ is uniformly bounded and obviously uniformly tight. Therefore, by Prokhorov's theorem $f_n (x, t, m) dm$ converges weakly to some Radon measure $\nu (x, t; dm)$. Thus, $s_j (f_n) (x, t, m) dm$ converges to $\int e^{-ij\theta} d\nu (x, t ; dm)$ for each $j$, but this equals to $s_j = \int_{\mathbb{S}^1} e^{-ij\theta} d\mu(\theta)$. Since $\mu$ is determined by the trigonometric moments, $\nu = \mu$. Note that $M_k [\mu] = M_k$ also holds.
\paragraph{$f$ is a weak solution of (\ref{Doi-Approx}).}
To show that $\mu$ is the weak solution of (\ref{Doi-Approx}), we write the Fokker-Planck equation of (\ref{Doi-Approx2}) in the weak form as in (\ref{Cauchy-FP}), and check the convergence. First, we note that
$$|E | =  \left | \{ t \in (0, T ) : \, \left | \{ x \in \mathbb{T}^2 : f_n (x, t, dm) \,\, \mathrm{does} \,\, \mathrm{not} \,\, \mathrm{converge} \,\, \mathrm{weakly} \,\, \mathrm{to} \,\, \mu(x, t; dm) \} \right | > 0 \} \right | = 0. $$
For $t \in (0, T) - E$, for almost all $x$, $\int_{\mathbb{S}^1} \phi_m (m) f_{n+1} (x, t, m) dm \rightarrow \int_{\mathbb{S}^1} \phi_m (m) \mu (x, t; dm)$ by weak convergence, and $\left | \int_{\mathbb{S}^1} \phi_m (m) f_{n+1} (x, t, m) dm \right | \le C_{\phi_m} B_4$ for almost every $x$, so by Dominated convergence theorem, 
$$ \int_{ \mathbb{T}^2 \times \mathbb{S}^1} \phi_x (x) \phi_m (m) f_{n+1} (x, t; m) dm dx \rightarrow  \int_{ \mathbb{T}^2 \times \mathbb{S}^1} \phi_x (x) \phi_m (m) \mu (x, t; dm) dx. $$
The second term is easy since the initial data of $f_n$ are just mollified ones of $f(0)$. We can show convergence for other terms except for the ones involving velocity field $u$, using the very same argument: by weak convergence we have almost everywhere convergence for $m$ integral part first, and for that term we have uniform bound (depending on $\phi_m$), then we apply dominated convergence theorem. For the terms involving velocity fields $u$, we apply the generalized dominated convergence theorem instead. Then standard parabolic regularity theory guarantees that actually $\mu$ is given by density $f(x, t, m) dm dx$, and if initial entropy is finite, then it remains finite, with $f(t) \in W^{1,1} (\mathbb{T}^2 \times \mathbb{S}^1)$, and $\int_0 ^T \int_{\mathbb{T}^2 \times \mathbb{S}^1} \frac{|\nabla_{x,m} f|^2}{f}dm dt < \infty$. 
\paragraph{Solution of (\ref{Doi}).} Existence of a solution of (\ref{Doi}) is just a repetition of arguments for establishing solutions of (\ref{Doi-Approx}). In this case, we set up $\ell \rightarrow \infty$.

\subsection{Uniqueness of the solution}

Uniqueness of the solution follows from relative energy method. Suppose that $(u, f)$ and $(v, g)$ are two solutions of (\ref{Doi}) with same initial data $(u_0, f_0)$ satisfying our assumptions. 
\paragraph{Control of $u-v$.}
By taking $L^2$ estimates, vorticity estimate, and $W^{1,2}$ norm estimates for $u-v$, we have
\begin{equation}
\begin{gathered}
\frac{d}{dt} \norm{u-v}_{W^{2,2}} ^2 + \norm{ (u-v)}_{W^{3,2}} ^2 \le C_1 (t) \norm{u-v}_{W^{2,2}} ^2  \\
+ C \norm{\sigma_E (f) - \sigma_E (g) }_{W^{2,2}} ^2 + C_2 (t) \norm{M_4 (f) - M_4 (g) }_{W^{2,2}} ^2
\end{gathered} \tag{Rel-Energy-Doi} \label{Rel-Energy-Doi}
\end{equation}
where $C_1 (t), C_2 (t) \in L^1 (0, T)$ coming from norms of $v$ and $C$ is a constant independent of time.
\paragraph{Control of $\int_{\mathbb{S}^1} |f-g| dm $.}
The key quantity of control is $\int_{\mathbb{S}^1} |f-g| dm $. Let $\mathrm{sgn}_\beta$ be a smooth, increasing regularization of the sign function such that $\mathrm{sgn}_\beta (s) = \mathrm{sign}(s) $ for $|s| \ge \beta$, and let $|s|_\beta = \int_0 ^s \mathrm{sgn}_\beta (r) dr$. Then as $\beta \rightarrow 0$, we have $|s |_\beta \rightarrow |s|$. Then by subtracting two Fokker-Planck equations of (\ref{Doi}) for $f$ and $g$, then by replacing $\phi_x (x) \phi_m (m)$ in (\ref{Cauchy-FP}) by $\mathrm{sgn}_\beta (f-g) \int |f-g|_\beta dm$ (we can do this since $C^\infty (\mathbb{T}^2 ) \otimes C^\infty (\mathbb{S}^1 )$ is dense in $L^p ( \mathbb{T}^2 \times \mathbb{S}^1, f(x, t, m) dx dm (g(x, t, m) dx dm) )$  for any $p \ge 1$), then by checking that terms from diffusion are positive, and finally taking the limit $\beta \rightarrow 0$ (and dividing by $\norm{\int_{\mathbb{S}^1}  |f-g| dm  }_{L^2} $ ), we obtain
\begin{equation}
\begin{gathered}
 \norm{\int_{\mathbb{S}^1}  |f-g| dm  }_{L^2} (t)  \le \int_0 ^t   \left ( \norm{u-v}_{W^{1,\infty}} \norm{\int_{\mathbb{S}^1} |\nabla_{x, m} g| dm }_{L^2} + C B_4 \norm{\nabla_x (u-v)}_{L^2} \right ) ds  .
\end{gathered} \tag{FPdiff} \label{FPdiff}
\end{equation}
Noting that 
\begin{equation*}
\begin{gathered}
\norm{\int_{\mathbb{S}^1} |\nabla g| dm }_{L^2} = \left  ( \int_{\mathbb{T}^2} \left ( \int_{\mathbb{S}^1} |\nabla g | dm \right )^2 dx \right )^{\frac{1}{2} } \le \left ( \int_{\mathbb{T}^2} B_4 \int_{\mathbb{S}^1} \frac{|\nabla g|^2} {g} dm dx  \right )^{\frac{1}{2} } 
\end{gathered}
\end{equation*}
we obtain
\begin{equation}
\begin{gathered}
 \norm{\int_{\mathbb{S}^1}  |f-g| dm  }_{L^2} (t)  \le \int_0 ^t  \left ( B_4 ^{\frac{1}{2} } \left ( \int_{\mathbb{T}^2 \times \mathbb{S}^1} \frac{| \nabla_{x, m} g |^2}{g} (s)  dm dx  \right )^{\frac{1}{2}} + C B_4 \right ) \norm{( u-v) (s) }_{W^{3,2} } ds \\
\le  C \norm{u-v}_{L^2 (0, t; W^{3,2} ) } \left ( \sqrt{t} + \left ( \int_0 ^t \int_{\mathbb{T}^2 \times \mathbb{S}^1 }  \frac{| \nabla_{x, m} g |^2}{g} (s)  dm dx ds \right )^{\frac{1}{2} } \right ) \le C (\sqrt{B_1} + \sqrt{t} ) \norm{u-v}_{L^2 (0, t; W^{3,2} ) }
\end{gathered} 
\end{equation}
thanks to the free energy estimate (\ref{FE}). Therefore,
\begin{equation}
\norm{\int_{\mathbb{S}^1} |f-g| dm (t) }_{L^2} ^2 \le C (1 + t) \norm{u-v}_L^2 (0, t; W^{3,2} ) ^2.
\tag{FPdiff2} \label{FPdiff2}
\end{equation}
\paragraph{Control of moments.}
Finally, we apply the relative energy estimates for evolution equation of moments (\ref{ME}), and apply (\ref{FPdiff2}) in closing the effect of higher moments to obtain the following:
\begin{equation}
\begin{gathered}
\frac{d}{dt} \norm{M_n (f) - M_n (g) }_{W^{1,2} } ^2 + \nu \norm{M_n (f) - M_n (g) }_{W^{2,2} } ^2 \\ 
\le C^n \norm{M_n (f) - M_n (g) }_{W^{1,2} } ^2 + C_n \norm{u-v}_{W^{1,2} } ^2 + C_3 (t) \norm{\int_{\mathbb{S}^1} |f-g| dm }_{L^2} ^2, \\
\frac{d}{dt} \norm{M_n (f) - M_n (g) }_{W^{2,2} } ^2 + \nu \norm{M_n (f) - M_n (g) }_{W^{3,2} } ^2 \\
\le C^n \norm{M_n (f) - M_n (g) }_{W^{2,2} } ^2 + C^n \norm{u-v}_{W^{2,2}} ^2 + \frac{\nu}{2} \norm{\Delta_x ( M_{n+2} (f) - M_{n+2} (g) )}_{L^2} ^2  \\
+ \frac{C^n}{\nu} \left ( \norm{\int_{\mathbb{S}^1} |f-g| dm }_{L^2} ^2 + \norm{\nabla_x ( M_{n+2} (f) - M_{n+2} (g) }_{L^2} ^2 \right ) ,
\end{gathered} \tag{Mdiff} \label{Mdiff}
\end{equation}
where $C_3 \in L^1 (0, T ) $ depending on the norms of $v$ and $g$, and $C^n$ are constants depending only on $n$ and norms of $v$ and $g$. Summing up (\ref{Rel-Energy-Doi}), (\ref{FPdiff2}), and (\ref{Mdiff}), we finally obtain
\begin{equation}
\begin{gathered}
\frac{d}{dt} F(t) + G(t) \le \phi(t) \left ( F(t) + \int_0 ^t G(s) ds \right ), \phi \in L^1 (0, T), G(t) = \norm{(u - v)(t) }_{W^{3,2} } ^2 , \\
F(t) = \norm{ u - v}_{W^{2,2} } ^2 + \norm{\sigma_E (f) - \sigma_E (g) }_{W^{2,2} } ^2 + \norm{M_4 (f) - M_4 (g) }_{W^{2,2} } ^2 + \norm{M_6 (f) - M_6 (g) }_{W^{1,2} } ^2 (t) 
\end{gathered}
\end{equation}
as desired. Noting that $F(0)= 0$ and applying Gr\"{o}nwall's inequality, we see that $F(t) = \int_0 ^t G(s) ds = 0$. This proves the uniqueness of the solution, and completes the proof of Theorem \ref{Main}.

\section{Conclusion}

We proved Theorem \ref{Main}, which states that the well-posedness of strong solutions of (\ref{Doi}) with free energy estimates. Also we showed well-posedness of strong solutions of (\ref{DA}), which is an approximate closure of (\ref{Doi}). In addition, we report that the condition $\eta < \eta_c$ for some critical $\eta_c$, which was used to control viscous stress term, was not necessary and we can show global well-posedness for all $\eta >0$. 
\paragraph{Acknowledgement} The author deeply appreciates the helpful support of Professor Peter Constantin. Professor Constantin encouraged the development of the paper and gave a lot of helpful comments. Doctor Dario Vincenzi is the one who introduced the author to work on (\ref{DA}), and the author also thanks to Doctor Vincenzi. Research of the author was partially supported by Samsung scholarship. 

\bibliographystyle{abbrv}
\bibliography{jl4}
\end{document}